\documentclass{article}%
\usepackage{amsmath}
\usepackage{amsfonts}
\usepackage{amssymb}
\usepackage{graphicx}
\usepackage{scalefnt}
\usepackage{tikz-cd}
\usepackage{adjustbox}%
\providecommand{\U}[1]{\protect\rule{.1in}{.1in}}
\newtheorem{theorem}{Theorem}
\newtheorem{acknowledgement}[theorem]{Acknowledgement}

\newtheorem{corollary}[theorem]{Corollary}

\newtheorem{definition}[theorem]{Definition}
\newtheorem{example}[theorem]{Example}

\newtheorem{lemma}[theorem]{Lemma}
\newtheorem{notation}[theorem]{Notation}

\newtheorem{proposition}[theorem]{Proposition}
\newtheorem{remark}[theorem]{Remark}

\newenvironment{proof}[1][Proof]{\noindent\textbf{#1.} }{\ \rule{0.5em}{0.5em}}
\begin{document}

\title{\bigskip Skeletal Homology: Singular Homology at Scale}
\author{Ivy Dey and Conrad Plaut\\Department of Mathematics\\The University of Tennessee\\Knoxville TN 37996\\idey@vols.utk.edu, cplaut@utk.edu}
\date{}
\maketitle

\begin{abstract}
By ``skeletal homology'' $K_{n}^{\varepsilon}(X)$ we mean the homology of the
chain complex $S_{n}^{\varepsilon}(X)$ generated by skeletal $(n,\varepsilon
)$-simplices, i.e.\ maps from the $0$-skeleton of the standard simplex into a
metric space $X$ whose image has diameter less than $\varepsilon$. Goldfarb
showed for finite metric spaces that this homology is isomorphic to the
simplicial homology $H_{n}^{\Delta}(VR_{\varepsilon}(X))$ of the
Vietoris--Rips complex. We give a new proof of this isomorphism for arbitrary
metric spaces. We then develop a metric analogue of methods from singular
homology by replacing the topological notion of ``homotopic'' with the metric
notion of ``close''. The key tool is a discrete prism construction, which
shows that uniformly close functions with a common modulus of continuity
induce the same homomorphism on skeletal homology. This leads to stability
results for
$K_{n}^{\varepsilon}(X)\cong H_{n}^{\Delta}(VR_{\varepsilon}(X))$, including
broad extensions of the theorems of Hausmann and Latschev beyond the setting
of Riemannian manifolds. Finally, repeated barycentric subdivision and
restriction induce a natural homomorphism $\rho_{\varepsilon}:H_{n}%
(X)\longrightarrow K_{n}^{\varepsilon}(X). $ We call its image the \emph{real
homology} $H_{n}^{\varepsilon}(X)$ at scale and define the corresponding
\emph{phantom homology} $P_{n}^{\varepsilon}(X) = K_{n}^{\varepsilon}%
(X)/H_{n}^{\varepsilon}(X). $ We show that phantom homology can detect
impending topological collapse, as illustrated by collapsing Riemannian
manifolds such as Berger spheres.

\end{abstract}

\section{Introduction}

What is now known as the Vietoris-Rips (VR) complex was first introduced by
Vietoris in 1927 \cite{V}, and was later independently defined and used by
Rips. The VR complex $VR_{\varepsilon}(X)$ takes as its simplices at scale
$\varepsilon>0$ all subsets of a metric space $X$ of diameter less than
$\varepsilon$. For nearly a century, the VR complex has been the most widely
used way to understand the homology of metric spaces at a given scale, along
with other proxy complexes, such as the \v{C}ech complex. We will refer to
these methods generally as \textit{topology by simplicial proxy}. Restating
the obvious, these methods are both \textit{indirect} and \textit{simplicial}.

Because topology by simplicial proxy is \textit{indirect}, some justification
is needed for the idea that these constructions actually \textquotedblleft
see\textquotedblright\ the topology of the ambient metric space. There are two
important justifications for this idea. The first is Hausmann's Theorem
(\cite{Hau}), which says that for a compact Riemannian manifold, at a small
enough scale, the simplicial homology of the VR complex (which we will abbreviate
as the VR homology) is isomorphic to the homology of the ambient manifold. The
second is Latschev's Theorem, which says that given a compact Riemannian
manifold $M$ and a particular scale $\varepsilon>0$, there exists some
$\delta_{\varepsilon}>0$ such that if $X$ is a compact metric space with
Gromov-Hausdorff distance $d_{GH}(X,M)<\delta_{\varepsilon}$ then $X$ and $M$
have the same VR homology. Another issue arising from indirectness is that the
proxy complexes tend to be huge; for example, the vertex set of the VR complex
is equal to the set of points in the space. This can create computational
difficulties for large data sets. For any non-trivial path-connected metric
space, the VR complex has uncountably many simplices in every dimension.

Because these methods are \textit{simplicial}, they share the same
deficiencies that motivated the wholesale shift from simplicial homology to
singular homology in algebraic topology. One of the most significant issues is
the fact that a continuous map between triangulated spaces may not be
simplicial, i.e., functorality is problematic. This issue manifests itself in
the study of VR-complexes because if $f:X\rightarrow Y$ is a function between
metric spaces, there may be no induced map on the VR complexes unless $f$ is,
say, $1$-Lipschitz. Yet many modern problems in geometry rely on functions
that are only \textquotedblleft continuous relative to scale\textquotedblright%
, such as quasi-isometries.

Some alternatives to topology by simplicial proxy have been considered in this
century, in chronological order: Discrete homotopy theory for the fundamental
group due to Berestovskii-Plaut-Wilkins, (\cite{BPTG}, \cite{BPUU},
\cite{PW1}, \cite{PW2}); discrete methods for Hodge and de Rham cohomology due
to Bartholdi-Schick-Smale-Smale (\cite{BSSS}); homology at scale in the
context of applied topology of finite data spaces due to Goldfarb (\cite{G});
homotopy at scale via \v{C}ech closure spaces due to Rieser (\cite{Rieser2021}%
); nine singular homology theories for closure spaces, including three
simplicial--cubical pairs, yielding six distinct theories in general due to
Bubenik-Mili\'{c}evi\'{c} (\cite{BM}); and additional approaches to homotopy
at scale due to M\'{e}moli-Zhou (\cite{MZ}). The homology considered by
Goldfarb for finite metric spaces is, for arbitrary metric spaces, isomorphic
to one of the homologies defined by Bubenik-Mili\'{c}evi\'{c}, although the
latter homology was defined in a completely different way. Although we do not
consider this question at present, it seems likely that this particular
homology is unique in some general sense, and we will refer to it as
\textit{skeletal homology}. This name emphasizes what we regard as its most
important feature:\ the finite domain of definition of the simplices.

For $n\in\mathbb{N}$ and $\varepsilon>0$, an $(n,\varepsilon)$%
-\textit{skeletal simplex }is function $\sigma$ from the $0$-skeleton $V^{n}$
of the standard simplex $\Delta^{n}$ into a metric space $X$, such that the
diameter of the image of $\sigma$ is less than $\varepsilon$. On the free
abelian group $S_{n}^{\varepsilon}(X)$ generated by $(n,\varepsilon)$-skeletal
simplices, a boundary map may be defined by restricting the usual boundary map
to $0$-skeleta. That is, $\partial_{\varepsilon}\sigma=\sum_{j=0}^{n}%
(-1)^{j}\sigma^{j}$, where $\sigma^{j}$ is the restriction of $\sigma$ to the
$0$-skeleton of the appropriate face. With this notion of boundary, the groups
$S_{n}^{\varepsilon}(X)$ form what we will call the \textit{skeletal chain
complex}. The homology of this chain complex will be called \textit{skeletal
homology }$K_{n}^{\varepsilon}(X)$. For $0<\delta<\varepsilon$, the inclusion
maps $i_{\varepsilon\delta}:S_{n}^{\delta}(X)\rightarrow S_{n}^{\varepsilon
}(X)$ are chain maps that induce change-of-scale homomorphisms $j_{\varepsilon
\delta}:K_{n}^{\delta}(X)\rightarrow K_{n}^{\varepsilon}(X)$, which together
form an inverse system. Equivalently, these maps form what is called a
\textit{persistence module} by those working in applied topology and
topological data analysis. We will use the language of persistence modules
when citing theorems using that language.

Goldfarb showed (\cite{G}) that for a \textit{finite} metric space $X$,
$K_{n}^{\varepsilon}(X)$ is isomorphic to the simplicial homology
$H_{n}^{\Delta}(VR_{\varepsilon}(X))$; see also \cite{MS} for an alternative
proof. We generalize this result to arbitrary metric spaces (Theorem
\ref{VR}), which extends the applicability of this isomorphmism to the realm of
metric geometry. While Goldfarb's proof relies on other papers in applied
topology, our proof in complete generality is entirely self-contained, and is
analogous to the classical proof of the equivalence of simplicial homology of
a simplicial complex $C$ with the singular homology of its underlying
topological space $\left\vert C\right\vert $. That proof proceeds by induction
on the simplicial $k$-skeleta of $C$, using the long exact sequence for pairs
from the $k$-skeleta. In our case, the role of the \textquotedblleft
underlying space\textquotedblright\ is played by $X$, which generally has no
simplicial structure. Therefore, we define a long exact sequence of
\textquotedblleft fake skeletal pairs\textquotedblright\ (\ref{2rr}) involving
the cardinalities of the images of skeletal simplices (Definition
\ref{fakepair}). We mention this detail because it is one of many instances in
which the finiteness of the domains of skeletal simplices plays an essential
role in our work. Our proof also reveals additional information about the
generators of $K_{n}^{\varepsilon}(X)$ that may have computational
consequences for finite data spaces (Remark \ref{VRR}).

There is a natural homomorphism $\rho_{\varepsilon}:H_{n}(X)\rightarrow
K_{n}^{\varepsilon}(X)$ from singular homology to skeletal homology, for all
$\varepsilon>0$. This map, which has no known analog in VR homology, is
induced by repeated barycentric subdivision followed by the restriction of
each continuous simplex to its $0$-skeleton. The map $\rho_{\varepsilon}$ is
the link between singular homology and skeletal homology, and plays a
fundamental role in this paper. For example, the standard method to show that
$H_{n}^{\Delta}(VR_{\varepsilon}(X))$ is isomorphic to $H_{n}(X)$ at some
scale, going back to Hausmann's Theorem, is to construct a homotopy
equivalence between $X$ and (the underlying space of) $VR_{\varepsilon}(X)$.
Instead, one may now simply show directly that $\rho_{\varepsilon}%
:H_{n}(X)\rightarrow K_{n}^{\varepsilon}(X)$ is an isomorphism. In this way we
prove a strong generalization of Hausmann's Theorem (Theorem \ref{HM2}) that
applies to any space that has a bounded \textquotedblleft contractibility
ratio\textquotedblright\ (Definition \ref{fcr}) in the sense that small enough
balls may be uniformly contracted in a ball that is larger by a fixed ratio.
This assumption includes spaces that are \textit{uniformly locally
contractible} in the sense that for some $\psi>0$, all balls of radius
$0<r<\psi$ are contractible. Therefore, this theorem applies to compact
Riemannian manifolds and many other spaces, including metric graphs and metric
cones. In fact, as we will explain in more detail later, Theorem \ref{HM2}
implies all variants of Hausmann's Theorem of which we are aware (\cite{O2},
\cite{AA2}, \cite{AM}). 

Of course constructing a homotopy equivalence is stronger than showing an
isomorphism of homology. On the one hand, it is unclear to us whether this
stronger conclusion is true under the much weaker assumption of Theorem
\ref{HM2}. On the other hand, one of the goals of this paper is to show that
by avoiding proxy complexes altogether, one can obtain stronger results about
homology at scale with much simpler arguments, often reminiscent of those in
classical singular homology. Indeed, our proof of Theorem \ref{HM2} takes
about two pages. We use the bounded contractibility ratio to \textquotedblleft
fill in\textquotedblright\ skeletal simplices, and barycentric subdivision to
reduce the size of the resulting continuous cycles. As in many classical
arguments, this process is algebraically formalized by suitable chain homotopies.

Hausmann's original theorem is qualitative in the sense that it asserts the
existence of a scale below which $H_{n}^{\Delta}(VR_{\varepsilon}(X))\equiv
H_{n}(X)$. It has since been observed (\cite{Ma}) that for a compact
Riemannian manifold, this isomorphism exists for any $\varepsilon>0$ less than
the convexity radius of $M$. We believe that our Theorem \ref{HM2} can be
improved for Riemannian manifolds through a careful consideration of the
fillings that arise from the exponential map, but we will not consider this
here. 

In \cite{LM}, Lim-M\'{e}moli-Okutana conjectured (Conjecture 9.5) that
Hausmann's Theorem could be generalized to compact metric ANRs. We will not
recall the definition of ANR here since we do not need it. Gillespie showed
(\cite{Gi}) that there are compact, locally contractible geodesic spaces
(which are known to be ANRs) for which Hausmann's Theorem fails, disproving
Conjecture 9.5. But as is easily checked, Gillespie's examples are not
\textit{uniformly} locally contractible, so they do not contradict Theorem
\ref{HM2}.

As for functorality, given a function $f:X\rightarrow Y$ between metric
spaces, possibly not continuous, we may define an induced chain map on the
skeletal chain complex, hence on skeletal homology, in the usual way. The
difference here is that the scales of the domain and range are determined by
an extended notion of \textit{modulus of continuity} (Definition \ref{moc} and
Remark \ref{mc}). We may then define concepts such as (possibly not
continuous) $\delta$-homotopy equivalence (Definition \ref{dhi} and see also
\cite{BM}) and $\delta$-deformation retraction between metric spaces.

A major theme of this paper is that \textquotedblleft
homotopic\textquotedblright\ in singular homology often corresponds to
\textquotedblleft close\textquotedblright\ in skeletal homology. Our main
stability tool (Theorem \ref{induced}) is the skeletal analogue of the
classical theorem that homotopic maps induce the same homomorphism on
homology: if two (possibly not continuous) functions $f,g:X\rightarrow Y$ with
a common modulus of continuity are uniformly within $\delta$ of one another,
then they induce the same homomorphism on skeletal homology, at the cost of
increasing the scale by $\delta$. The proof is a discrete version of the
classical prism construction. Here again, the finiteness of the domains of
skeletal simplices is essential: the \textquotedblleft
prism\textquotedblright\ over a skeletal simplex is simply a finite set of
points, so no continuity is needed to fill it. All of our stability-type
theorems follow relatively easily from Theorem \ref{induced}.

Latschev-type theorems are essentially a combination of a Hausmann-type
theorem and a stability theorem for homology at scale. In fact, Oudot-Waas
\cite{OW} give a formal process for deriving Latschev-type theorems, which may
be applicable to our results. However, because the proofs are short and
self-contained, we directly show a generalization of Latschev's theorem in our
setting (Theorem \ref{qlatschev}) for uniformly locally contractible spaces.

As Majhi commented in \cite{Ma}, Latschev's theorem \textquotedblleft has been
regarded as a stepping stone to the finite reconstruction of an abstract
Riemannian manifold from a noisy sample\textquotedblright. However, the
assumption that the underlying space is a Riemannian manifold seems a bit
strong. For example, what if the sampled space happens to be a metric graph,
which is probably more easily analyzed than some approximation of it by a
Riemannian manifold? Theorem \ref{qlatschev} applies to a much wider class of
spaces, including geodesic graphs.

By Theorem \ref{HM2}, for a uniformly locally contractible space, singular
homology and skeletal homology are isomorphic at small enough scales. The
natural question arises: how are skeletal and singular homology related at
larger scales? To understand the difference between these homologies, we let
$H_{n}^{\varepsilon}(X)$ denote the image $\rho_{\varepsilon}(H_{n}(X))\subset
K_{n}^{\varepsilon}(X)$, which we call the \textit{real homology} of $X$ at
the scale of $\varepsilon>0$ (Definition \ref{rf}). We define \textit{phantom
homology} to be $P_{n}^{\varepsilon}(X)=K_{n}^{\varepsilon}(X)/H_{n}%
^{\varepsilon}(X)$. Phantom homology is not necessarily a bad thing. For
example, discrete metric spaces have only phantom homology for $n\geq1$.
Phantom homology can also \textquotedblleft anticipate\textquotedblright\ the
abrupt changes in singular homology that can occur in Gromov-Hausdorff
convergence (Example \ref{berg}). But if one is interested in understanding
\textit{bona fide }homology at scale, we believe that $H_{n}^{\varepsilon}(X)$
provides a more faithful description than $H_{n}^{\Delta}(VR_{\varepsilon
}(X))$--see Diagram (\ref{rhdia}) and the subsequent comments.

\begin{remark}
Theorems \ref{VR} and \ref{HM2} are stated for coefficients in $\mathbb{Z}$,
but the proofs are the same when the coefficients are in any abelian group.
The same is true of Theorem \ref{induced} and the stability results that follow
from it, since the prism construction is purely combinatorial. For
simplicity, our statements and proofs are for coefficients in $\mathbb{Z}$.
\end{remark}

Despite the issues that we have described, topology by simplicial proxy has
played a significant role in applied topology and topological data analysis
(\cite{EDEL}, \cite{CV}, among many good general references), as well as
metric geometry and geometric group theory (\cite{AA}, \cite{AA2}, \cite{AFV},
\cite{LM}, \cite{BCM}). We believe that the methodology of skeletal homology
may have additional applications in these areas of mathematics. For example,
the fact that skeletal simplices have finite domains suggests that there may
be relatively efficient ways to directly calculate it for finite metric
spaces, including data. We explore these approaches in \cite{SH2}. We list
some open questions at the end of the next section.

\section{Main Results}

To have a common reference point for classical singular homology, we will
generally use adaptations of the notation and conventions in Hatcher's book
\cite{H}. Let $X$ be a metric space. The \textit{skeletal chain complex} for
$\varepsilon>0$ consists of
\[
\cdots\longrightarrow S_{n+1}^{\varepsilon}%
(X)\xrightarrow{\partial_{\varepsilon}}S_{n}^{\varepsilon}%
(X)\xrightarrow{\partial_{\varepsilon}}\cdots
\xrightarrow{\partial_{\varepsilon}}S_{0}^{\varepsilon}(X)\longrightarrow0
\]
where each $S_{n}^{\varepsilon}(X)$ is generated by skeletal simplices
$\sigma:V^{n}\rightarrow X$ such that the diameter of the image is less than
$\varepsilon$. The boundary $\partial_{\varepsilon}$ is defined by letting
$\sigma^{j}$ be the restriction of $\sigma$ to the subset $\{v_{0}%
\ldots,\widehat{v_{j}},\ldots,v_{n}\}$, and letting $\partial_{\varepsilon}$
be the linear extension to $S_{n+1}^{\varepsilon}(X)$ of $\partial
_{\varepsilon}\sigma=\sum_{j=0}^{n}(-1)^{j}\sigma^{j}$. Strictly speaking, in
order for the face $\sigma^{j}$ to be a skeletal $\varepsilon$-simplex, it
must be pre-composed with the order-preserving linear map $i_{j}$ from the
$0$-skeleton of standard simplex $\Delta^{n-1}$ to the face that is the domain
of $\sigma^{j}$ (see Hatcher's comment on p. 108). Including this
pre-composition in the notation, i.e. $\partial\sigma=\sum_{j=0}^{n}%
(-1)^{j}\sigma^{j}\circ i_{j}$, complicates the exposition and is inessential
to the proofs. The fact that $\partial_{\varepsilon}^{2}=0$ follows from the
standard argument, which is purely combinatorial and involves restrictions to
the faces $\sigma^{j}$, which are $(\varepsilon,n-1)$-skeletal simplices. In
fact, the purely \textquotedblleft combinatorial/algebraic
parts\textquotedblright\ of proofs from classical singular homology often
carry over to this setting.

We will refer to the elements of $S_{n}^{\varepsilon}(X)$ as $\varepsilon
$\textit{-chains} and elements of the kernel of $\partial_{\varepsilon}$ as
$\varepsilon$\textit{-cycles}. Except where needed for clarity, we will
generally denote $\partial_{\varepsilon}$ by $\partial$. The skeletal homology
class of an $\varepsilon$-cycle $c$ will be denoted by $[c]_{\varepsilon}\in
K_{n}^{\varepsilon}(X)$. In \cite{PW1} the term \textquotedblleft$\varepsilon
$-chain\textquotedblright\ and notation $[\lambda]_{\varepsilon}$ refer to
concepts related to the groups $\pi_{\varepsilon}(X,\ast)$, which we will
discuss in more detail later. However, the two contexts are entirely
different, and there should be no confusion. Also, for simplicity, we will
only prove our theorems for $n\geq1$, except for a couple of inductive proofs
that begin with $0$-simplices. $S_{0}^{\varepsilon}(X)$ is the free abelian
group generated by constant maps, and $K_{0}^{\varepsilon}(X)=H_{0}%
^{\varepsilon}(X)$ is the free abelian group generated by the so-called
\textquotedblleft$\varepsilon$-components\textquotedblright. The exposition of
singular $0$-homology readily carries over to skeletal $0$-homology (\cite{DD}).

For every $n$ and $0<\delta<\varepsilon$, the inclusion of $S_{n}^{\delta}(X)$
into $S_{n}^{\varepsilon}(X)$ is a chain map that induces a \textquotedblleft
change of scale\textquotedblright\ homomorphism $j_{\varepsilon\delta}%
:K_{n}^{\delta}(X)\rightarrow K_{n}^{\varepsilon}(X)$. Explicitly,
$j_{\varepsilon\delta}([c]_{\delta})=[c]_{\varepsilon}$. Note that if $c$ is a
$\delta$-cycle then $c$ is also an $\varepsilon$-cycle for any $\varepsilon
>\delta$, and it is possible that $c$, considered as an $\varepsilon$-cycle,
is the boundary of an $\varepsilon$-chain. This explicitly means that
$[c]_{\delta}\in\ker j_{\varepsilon\delta}$ and we will call $c$ an
$\varepsilon$-boundary.

We will refer to the singular homology groups of $X$ as $H_{n}(X)$. Since we
are working with simplices that are even more \textquotedblleft
singular\textquotedblright\ than continuous ones, we will often simply use the
adjective \textquotedblleft continuous\textquotedblright\ when referring to
classical singular simplices and chains. The singular homology class of a
continuous cycle will be denoted by $[c]$. For every $\varepsilon>0$, let
$H_{S,n}^{\varepsilon}(X)$ denote the homology of the chain complex
$C_{S,n}^{\varepsilon}(X)$, where the generators of $C_{S,n}^{\varepsilon}(X)$
are continuous simplices having images of diameter less than $\varepsilon$. An
unsurprising tweak of a classical result (Theorem \ref{rh}) says the
inclusion-induced map $\beta_{\varepsilon}:H_{S,n}^{\varepsilon}(X)\rightarrow
H_{n}(X)$ is an isomorphism. That is, restricting the \textquotedblleft
size\textquotedblright\ of continuous simplices produces nothing new in
singular homology. A key to understanding the relationship between singular
homology and skeletal \ homology is the restriction map $r$, defined for any
continuous simplex $\sigma$ by $r(\sigma)=\sigma\mid V^{n}$. For any fixed
$\varepsilon>0$, $r$ extends to a chain map $r_{\varepsilon}:C_{S,n}%
^{\varepsilon}(X)\rightarrow S_{n}^{\varepsilon}(X)$, which induces a
homomorphism $r_{\varepsilon}^{\ast}:H_{S,n}^{\varepsilon}(X)\rightarrow
K_{n}^{\varepsilon}(X)$.

\begin{definition}
\label{rf}We denote the map $r_{\varepsilon}^{\ast}\circ\beta_{\varepsilon
}^{-1}$ by $\rho_{\varepsilon}:H_{n}(X)\rightarrow K_{n}^{\varepsilon}(X)$.
(As is often done in classical homology and as we will do throughout this
paper, to avoid notational clutter, the \textquotedblleft$n$\textquotedblright%
\ in many expressions will often be suppressed, as will an occasional
\textquotedblleft$\varepsilon$\textquotedblright\ in designation of induced
maps and the boundary.) The image $H_{n}^{\varepsilon}(X)$ of $\rho
_{\varepsilon}$ is called the \textit{real homology} of $X$ (at scale
$\varepsilon$). The quotient group $P_{n}^{\varepsilon}(X)=K_{n}^{\varepsilon
}(X)/H_{n}^{\varepsilon}(X)$ is called the \textit{phantom homology of }$X$.
\end{definition}

\begin{theorem}
\label{cpd}For every $n$ and $0<\delta<\varepsilon$ we have the following
commutative diagram
\begin{equation}
\begin{tikzcd}[column sep=2.0em, row sep=2.0em] 0 \arrow[r] & H_{n}^{\delta}(X) \arrow[r, "i_{\delta}"] \arrow[d, "\eta_{\varepsilon\delta}"'] & K_{n}^{\delta}(X) \arrow[r, "\pi_{\delta}"] \arrow[d, "j_{\varepsilon\delta}"'] & P_{n}^{\delta}(X) \arrow[r] \arrow[d, "\phi_{\varepsilon\delta}"] & 0 \\ 0 \arrow[r] & H_{n}^{\varepsilon}(X) \arrow[r, "i_{\varepsilon}"'] & K_{n}^{\varepsilon}(X) \arrow[r, "\pi_{\varepsilon}"'] & P_{n}^{\varepsilon}(X) \arrow[r] & 0 \end{tikzcd} \label{is}%
\end{equation}
where each horizontal sequence is exact, and the vertical (\textquotedblleft
change of scale\textquotedblright) maps form an inverse system indexed on the
real numbers with order reversed. Here $\pi_{\varepsilon}$ and $\pi_{\delta}$
are the quotient maps, $i_{\varepsilon}$ and $i_{\delta}$ are inclusion maps
and the maps $\phi_{\varepsilon\delta}$ are induced by $j_{\varepsilon\delta}%
$. Moreover, the maps $\eta_{\varepsilon\delta}$ are surjective.
\end{theorem}

Recall that by definition, a collection of maps $\{f_{\varepsilon\delta}\}$ is
an inverse system if $f_{\varepsilon\sigma}=f_{\varepsilon\delta}\circ
f_{\delta\sigma}$ whenever $0<\sigma<\delta<\varepsilon$. We choose to express
the maps as an inverse system with parameters going from large to small, as
opposed to a direct system with parameters going from small to large, because
for path-connected spaces, as $\varepsilon\rightarrow0$, more and more
homology \textquotedblleft becomes visible\textquotedblright. For spaces of
infinite topological type, such as the Hawaiian Earring, the inverse limit may
be of interest as a kind of \textquotedblleft shape
homology\ group\textquotedblright. As previously mentioned, these vertical
sequences also comprise\textit{ persistence modules, }which are functors from
partially ordered sets into systems of algebraic objects with precisely the
same compatibility condition as an inverse limit. We will use the inverse
limit and persistence module notations interchangeably, depending on the context.

The above diagram shows that skeletal homology is a bridge between classical
singular homology (on the left) and homology of discrete spaces, which have
only phantom homology (on the right) for $n\geq1$. One can now clearly see the
difference between $H_{n}^{\varepsilon}(X)$ and $H_{S,n}^{\varepsilon}(X)$.
Both groups are generated by cycles with continuous representatives, but for
$H_{n}^{\varepsilon}(X)$ the chains that cycles may bound are permitted to be
skeletal, and therefore do not \textquotedblleft see topological holes smaller
than the given scale\textquotedblright. 

At small enough scales for nice enough spaces, there is no phantom homology.
We prove a somewhat more general statement than Theorem \ref{HM2}, using the
following definition. Recall that we denote by $\delta(S)$ the diameter of a
subset $S$ of $X$.

\begin{definition}
\label{fcr}Let $X$ be a metric space. A filling of a skeletal simplex $\sigma$
is defined to be any $\widehat{\sigma}$ such that $r_{\#}(\widehat{\sigma
})=\sigma$. We say that $X$ has a finite contractibility ratio if there is a
pair $(\phi_{F},\rho)$, with $\phi_{F}>0$ and finite $\rho\geq2$ (the
contractibility ratio), such that for every $x\in X$ and every $0<r<\phi_{F}$,
the ball $B(x,r)$ is contractible inside $B(x,\rho r/2)$. That is, the
inclusion $B(x,r)\hookrightarrow B(x,\rho r/2)$ is homotopic in $B(x,\rho r)$
to the constant map at $x$. We call $(\phi_{F},\rho)$ a fill pair for $X$.
\end{definition}

The contractibility ratio controls how much larger one makes a skeletal
simplex when extending it to a continuous simplex via a local contraction. If
$X$ has uniform contractibility radius $\psi_{0}>0$, then via the singular
cone construction and the triangle inequality, $X$ has a fill pair $(\psi
_{0},2)$. Therefore, Theorem \ref{HM2} implies Theorem \ref{HM2}.

\begin{theorem}
\label{HM2} Suppose that $X$ is a metric space with fill pair $(\phi_{F}%
,\rho)$. Then $X$ has positive stability radius at least $\rho^{-(n+1)}%
\phi_{F}$. Equivalently, if $0<\varepsilon<\rho^{-(n+1)}\phi_{F},$ then
$\rho_{\varepsilon}:H_{n}(X)\longrightarrow K_{n}^{\varepsilon}(X)$ is an
isomorphism.The following generalization of Hausmann's Theorem (at the level
of homology) says that there is no phantom homology at small enough scales for
spaces that are locally sufficiently nice:
\end{theorem}

\begin{theorem}
\label{HM}Let $X$ be a uniformly locally contractible metric space that is
Then for every $n$, if $0<\varepsilon<2^{-(n+1)}\psi_{0}$ the map
$\rho_{\varepsilon}:H_{n}(X)\rightarrow K_{n}^{\varepsilon}(X)$ is an isomorphism.
\end{theorem}

\bigskip

The link between singular homology and homology of discrete spaces is
strengthened by the following theorem. If $X$ is a metric space and $\delta
>0$, we denote the Vietoris-Rips complex of $X$ at the scale of $\delta$ by
$VR_{\delta}(X)$.

\begin{theorem}
\label{VR}If $X$ is any metric space, then for all $n$ and $\varepsilon>0$,
there is a natural isomorphism $\kappa_{\ast}:H_{n}^{\Delta}(VR_{\varepsilon
})\rightarrow K_{n}^{\varepsilon}(X)$ from the simplicial homology of
$VR_{\varepsilon}(X)$ to the skeletal homology of $X$ such that the following
diagram commutes for all $n$ and $0<\delta<\varepsilon$:
\begin{equation}
\begin{tikzcd}[column sep=2.0em, row sep=2.0em] H_{n}^{\Delta}(VR_{\delta}) \arrow[r, "f_{\varepsilon\delta}"] \arrow[d, "\kappa_{\ast}"'] & H_{n}^{\Delta}(VR_{\varepsilon}) \arrow[d, "\kappa_{\ast}"] \\ K_{n}^{\delta}(X) \arrow[r, "j_{\varepsilon\delta}"] & K_{n}^{\varepsilon}(X) \end{tikzcd} \label{dr}%
\end{equation}
Put another way, the maps $\kappa_{\ast}$ comprise an isomorphism of
persistence modules.
\end{theorem}

\begin{remark}
\label{sd} The map $f_{\varepsilon\delta}$ is induced by the inclusion of
simplicial complexes $\mathrm{VR}_{\delta}(X)\hookrightarrow\mathrm{VR}%
_{\varepsilon}(X)$, whereas $j_{\varepsilon\delta}$ is induced by the
inclusion of the skeletal chain complexes $S_{n}^{\delta}(X)\hookrightarrow
S_{n}^{\varepsilon}(X)$. The point is that skeletal homology gives a different
chain-level model for the same scale-dependent homology, in a way closer in
spirit to singular homology.
\end{remark}

The commutative diagram (\ref{dr}) allows one to see the relationship between
persistence diagrams of the VR-homology and skeletal critical values, which we
will define more precisely later in this introduction. Roughly speaking,
skeletal critical values occur when the maps $j_{\varepsilon\delta}$ are not
injective or not surjective. Non-surjectivity of the maps $j_{\varepsilon
\delta}$ results in additional generators (or basis elements with coefficients
in a field), which may be considered as \textquotedblleft
birth\textquotedblright. Those generators persist until they find themselves
in the kernel of some $j_{\varepsilon\delta}$, which results in their
\textquotedblleft death\textquotedblright. In \cite{DPU}, we explore, for
finite metric spaces, the computational implications of Theorem \ref{VR} and
other theoretical results about skeletal homology from the present paper.

\begin{notation}
If $f,g:X\rightarrow Y$ are bounded functions into a metric space $Y$, we
denote the uniform metric between them by $\left\vert f-g\right\vert
=\underset{x\in X}{\sup}\{d(f(x),g(x))\}$.
\end{notation}

\begin{definition}
\label{moc}If $f:X\rightarrow Y$ is a (possibly not continuous) function
between metric spaces, a \textit{modulus of continuity} $\omega$ of $f$ is a
non-decreasing function $\omega:[0,\infty)\rightarrow\lbrack0,\infty)$ such
that if $x,y\in X$ satisfy $d(x,y)<\omega(\varepsilon)$ then
$d(f(x),f(y))<\varepsilon$.
\end{definition}

If $\omega$ is a modulus of continuity of $f$, and both $\varepsilon$ and
$\omega(\varepsilon)$ are positive, it is immediate that the map
$\sigma\mapsto f\circ\sigma$ extends to a chain map $f_{\#}^{\omega}%
:S_{n}^{\omega(\varepsilon)}(X)\rightarrow S_{n}^{\varepsilon}(Y)$, which
induces a map $f_{\ast}^{\omega}:K_{n}^{\omega(\varepsilon)}(X)\rightarrow
K_{n}^{\varepsilon}(Y)$. Note that $f_{\ast}^{\omega}$ is functorial in the
sense that if $f_{i}$ has modulus of continuity $\omega_{i}$ then $\omega
_{1}\circ\omega_{2}$ is a modulus of continuity for $f_{2}\circ f_{1}$, and
$\left(  f_{2}\right)  _{\ast}^{\omega_{2}}\circ(f_{1})_{\ast}^{\omega_{1}%
}=\left(  f_{2}\circ f_{1}\right)  _{\ast}^{\omega_{1}\circ\omega_{2}}$. We
will often simply denote $f_{\ast}^{\omega}$ by $f_{\ast}$ and only include
$\omega$ in the notation for the domain and range of $f_{\ast}$.

\begin{remark}
\label{mc}Modulus of continuity is typically defined for uniformly continuous
functions, and typically is a positive function on $(0,\infty)$. However,
continuity is too strong an assumption for our purposes. Our more general
usage leads to some observations. First, $\underset{\varepsilon\rightarrow
0}{\lim}\omega_{f}(\varepsilon)$ need not be $0$. Suppose that $f$ is
$(\varepsilon,\delta)$-continuous, meaning that if $d(x,y)<\delta$ then
$d(f(x),f(y))<\varepsilon$. Then, equivalently, $f$ has a modulus of
continuity that is $0$ on the interval $[0,\varepsilon)$ and $\omega
(\varepsilon)=\delta$ otherwise. Clearly, every function has the $0$-function
as a modulus of continuity (the definition is vacuously true). Being
$1$-Lipschitz is equivalent to having the function $\omega(\varepsilon
)=\varepsilon$ as modulus of continuity, and for $1$-Lipschitz functions we
will always use this modulus of continuity. Any set of functions with the same
domain and range has a common modulus of continuity, namely the infimum of
their moduli of continuity (which could the constant map $0$). The only
complication of allowing $\omega(\varepsilon)=0$ for positive $\varepsilon$ is
that we will need to assume $\omega(\varepsilon)>0$ for expressions like
$K_{n}^{\omega(\varepsilon)}(X)$, as in the next theorem.
\end{remark}

Theorem \ref{induced} is analogous to the classical theorem that homotopic
maps induce the same homomorphism on homology, where, as mentioned above,
\textquotedblleft homotopic\textquotedblright\ is replaced by
\textquotedblleft close\textquotedblright. Unsurprisingly, there is some
\textquotedblleft loss\textquotedblright\ in this statement, quantified by
the change of scale from $\varepsilon$ to $\varepsilon+\delta$. Theorem
\ref{induced} is the primary tool for almost all of the remaining results in
this paper. Its proof, via a discrete prism construction, is given in Section
\ref{stabsec}; the proofs of the consequences below are short enough that we
include them in this section.

\begin{theorem}
\label{induced}Let $f,g:X\rightarrow Y$ be functions with a common modulus of
continuity $\omega$. If $0<\varepsilon,\delta$ are such that $\omega
(\varepsilon)>0$ and $\left\vert f-g\right\vert <\delta$ then $f_{\ast
}=g_{\ast}:K_{n}^{\omega(\varepsilon)}(X)\rightarrow K_{n}^{\varepsilon
+\delta}(Y)$.
\end{theorem}

Theorem \ref{induced} and functorality allow us to obtain many stability results.

\begin{definition}
\label{dhi} Let $X$ and $Y$ be metric spaces. Functions $f:X\to Y$ and $g:Y\to
X$ are called $\delta$-homotopy inverses for some $\delta>0$ if

\begin{enumerate}
\item $\omega(\varepsilon)=\max\{\varepsilon-2\delta,0\} $ is a common modulus
of continuity for both $f$ and $g$, and

\item $\left|  f\circ g-\operatorname{id}_{Y}\right|  , \left|  g\circ
f-\operatorname{id}_{X}\right|  <2\delta. $
\end{enumerate}

We will also refer to the pair $f,g$ as a $\delta$-homotopy equivalence.
\end{definition}

\begin{remark}
Although we will not use the idea in this paper, so we won't define it, we
note that $\delta$-homotopy equivalences form a \textit{quasi-isometry}.
\end{remark}

The \textit{interleaving distance} (\cite{CC}) in our context is defined as
follows. Given two metric spaces $X,Y$, $\{K_{n}^{\varepsilon}%
(X),j_{\varepsilon\delta}\}$ and $\{K_{n}^{\varepsilon}(Y),j_{\varepsilon
\delta}\}$ are said to be $\delta$-interleaved for some $\delta>0$ if there
exist a families of homomorphisms $\phi_{\varepsilon}:K_{n}^{\varepsilon
}(X)\rightarrow K_{n}^{\varepsilon+\delta}(Y)$ and $\psi_{\varepsilon}%
:K_{n}^{\varepsilon}(Y)\rightarrow K_{n}^{\varepsilon+\delta}(X)$ such that
$\psi_{\varepsilon+\delta}\circ\phi_{\varepsilon}=j_{\varepsilon
+2\delta,\varepsilon}$ and $\phi_{\varepsilon+\delta}\circ\psi_{\varepsilon
}=j_{\varepsilon+2\delta,\varepsilon}$. The \textit{interleaving}
\textit{distance} between the skeletal homology persistence modules $X$ and
$Y$ is the infimum of all $\delta>0$ such that $\{K_{n}^{\varepsilon
}(X),j_{\varepsilon\delta}\}$ and $\{K_{n}^{\varepsilon}(Y),j_{\varepsilon
\delta}\}$ are $\delta$-interleaved.

\begin{theorem}
\label{interleaf}If $f:X\rightarrow Y$ and $g:Y\rightarrow X$ are a $\delta
$-homotopy equivalence, then the induced maps $f_{\ast}$ and $g_{\ast}$
comprise a $3\delta$-interleaving between the skeletal homology persistence
modules of $X$ and $Y$.
\end{theorem}

\begin{proof}
From Theorem \ref{induced} we immediately have $(g\circ f)_{\ast}=\left(
id_{X}\right)  _{\ast}$ and $(f\circ g)_{\ast}=\left(  id_{Y}\right)  _{\ast}%
$. However, $(id_{X})_{\ast}$ and $(id_{Y})_{\ast}$ are not equal to the
identity! For one thing, the domain and range of each function are different.
In fact, $f\circ g$ and $g\circ f$ have modulus of continuity $\omega
(\varepsilon)=\max\{0,\varepsilon-4\delta\}$ and therefore the induced maps
have the required domain and range to be a $3\delta$-interleaving:
\[
(g\circ f)_{\ast}=g_{\ast}\circ f_{\ast}=\left(  id_{X}\right)  _{\ast
}=j_{\varepsilon+2\delta,\varepsilon-4\delta}:K_{n}^{\varepsilon-4\delta
}(X)\rightarrow K_{n}^{\varepsilon+2\delta}(X)
\]
and
\[
(f\circ g)_{\ast}=f_{\ast}\circ g_{\ast}=\left(  id_{Y}\right)  _{\ast
}=j_{\varepsilon+2\delta,\varepsilon-4\delta}:K_{n}^{\varepsilon-4\delta
}(Y)\rightarrow K_{n}^{\varepsilon+2\delta}(Y)
\]

\end{proof}

Recall that compact metric spaces $X,Y$ have Gromov-Hausdorff distance
$d_{GH}(X,Y)<\delta$ if there are isometric embeddings of $X$ and $Y$ into a
metric space $Z$, the images of which have Hausdorff distance $d_{H}(X,Y)<$
$\delta$. Gromov showed that this definition is equivalent to the following:
there is a relation $R$ on $X\times Y$ with distortion less than
$dis(R)<2\delta$. Here
\[
dis(R)=\sup\{\left\vert d(x,x^{\prime})-d(y,y^{\prime})\right\vert
:(x,y),(x^{\prime},y^{\prime})\in R\}
\]

\bigskip Suppose that $X,Y$ are subsets of a metric space with $d_{H}(X,Y)<$
$\delta$. We may define $f:X\rightarrow Y$ by choosing, for each $x\in X$,
some $f(x)\in Y$ such that $d(x,f(x))<\delta$. Similarly, define
$g:Y\rightarrow X$. Then by the triangle inequality, $f$ and $g$ are $\delta
$-homotopy inverses. Conversely, for $X,Y$ compact, suppose that $f$ and $g$
are $\delta$-homotopy inverses. Define a relation $R$ on $X\times Y$ by
$(x,y)\in R$ if and only $d_{Y}(f(x),y)<2\delta\text{ and }d_{X}%
(x,g(y))<2\delta$. By the second condition of the above definition, $R$ is a
correspondence. Now take $(x,y),(x^{\prime},y^{\prime})\in R$. Since $f$ has
modulus of continuity $\omega(\varepsilon)=\operatorname{max}\{0,\varepsilon
-2\delta\}$, we have $d_{Y}(f(x),f(x^{\prime}))\leq d_{X}(x,x^{\prime
})+2\delta$. Therefore
\[
d_{Y}(y,y^{\prime})\leq d_{Y}(y,f(x))+d_{Y}(f(x),f(x^{\prime}))+d_{Y}%
(f(x^{\prime}),y^{\prime})
\]%
\[
<2\delta+\big(d_{X}(x,x^{\prime})+2\delta\big)+2\delta=d_{X}(x,x^{\prime
})+6\delta
\]
Similarly, using the modulus of continuity for $g$, we obtain $d_{X}%
(x,x^{\prime})<d_{Y}(y,y^{\prime})+6\delta$. Hence $\left\vert d_{X}%
(x,x^{\prime})-d_{Y}(y,y^{\prime})\right\vert <6\delta$ for all
$(x,y),(x^{\prime},y^{\prime})\in R$. Thus $\operatorname{dis}(R)<6\delta$. We
summarize this discussion as the following lemma, which is unsurprising if not
precisely stated in metric geometry:\ 

\begin{lemma}
\label{ill}If $X$ and $Y$ are subsets of a metric space $Z$ and $d_{H}(X,
Y)<\delta$ for some $\delta>0$, then $X$ and $Y$ are $\delta$-homotopy
equivalent. Conversely, if $X$ and $Y$ are compact and $\delta$-homotopy
equivalent for some $\delta>0$, then $d_{GH}(X,Y)<3\delta$.
\end{lemma}

We obtain a new proof of interleaving stability with respect to the
Gromov-Hausdorff metric (Lemma 4.3 \cite{CDO}).

\begin{corollary}
\label{ilc}If $X$ and $Y$ are compact metric spaces with $d_{GH}(X,Y)<\delta$
then the interleaving distance between their skeletal homology persistence
modules is less than $2\delta$.
\end{corollary}

Compactness is not required for some of the results that follow, and therefore
we will state them in terms of the Hausdorff distance, with immediate
consequences for the Gromov-Hausdorff distance when the spaces are compact.

Assuming $f$ and $g$ are $\delta$-homotopy inverses (assume $\varepsilon
>2\delta$), we have the following commutative diagram, ignoring $h$ for the
moment:
\begin{equation}
\begin{tikzcd}[column sep=2.4em, row sep=2.4em] & K_{n}^{\varepsilon-2\delta}(Y) \arrow[d, "j_{\varepsilon,\varepsilon-2\delta}"'] \arrow[dl, "g_{\ast}"'] \\ K_{n}^{\varepsilon}(X) \arrow[r, dashed, "h"] \arrow[dr, "f_{\ast}"'] & K_{n}^{\varepsilon}(Y) \arrow[d, "j_{\varepsilon+2\delta, \varepsilon}"'] \\ & K_{n}^{\varepsilon+2\delta}(Y) \end{tikzcd} \label{HI}%
\end{equation}

Suppose that $j_{\varepsilon+2\delta,\varepsilon}$ and $j_{\varepsilon
,\varepsilon-2\delta}$, and hence $j_{\varepsilon+2\delta,\varepsilon-2\delta
}$, are isomorphisms. Then we may define $h:=(j_{\varepsilon+2\delta
,\varepsilon})^{-1}\circ f_{\ast}$, so that the diagram still commutes, and a
quick diagram chase shows that $h$ is also an isomorphism, and therefore also
are $g_{\ast}$ and $f_{\ast}$. In particular, $K_{n}^{\varepsilon}(X)$ is
isomorphic to $K_{n}^{\varepsilon}(Y)$. This shows the importance of knowing
when $j_{\varepsilon\delta}$ is an isomorphism, and motivates the following definition:

\begin{definition}
\label{critical}Let $X$ be a metric space. A number $\psi>0$ is called
homology (resp. skeletal) non-critical for some $n\geq0$ if there exists some
$\kappa>0$, such that for all $0<\delta<\kappa$, $\eta_{\psi+\delta
,\psi-\delta}$ is an isomorphism (resp. $j_{\psi+\delta,\psi-\delta}$) is an
isomorphism. Otherwise, $\psi$ is called \textit{homology (resp. skeletal)}
critical. The set of all homology (resp. skeletal) critical values is called
the homology (resp. skeletal) critical spectrum $\eta_{n}(X)$ (resp.
$\kappa_{n}(X)$). An interval that contains no homology (resp. skeletal)
critical values is called homology (resp. skeletal) non-critical. For
technical reasons (see below) we allow a homology non-critical interval to
have a negative lower endpoint.
\end{definition}

Since the maps $\eta_{\varepsilon\delta}$ are surjective, it follows from
Diagram (\ref{is}) that the homology critical spectrum is contained in the
skeletal critical spectrum. By definition, the set of homology non-critical
values and the set of skeletal non-critical values are both open, hence the
two critical spectra are closed in $(0,\infty)$. Note that if $D$ is a finite
metric space, then $\kappa_{n}(D)$ is finite. In fact, the set of possible
distances between points is finite. If there are no distance values strictly
between $\psi_{1}<\psi_{2}$ then for any $0<\psi_{1}<\delta<\varepsilon
<\psi_{2}$, $S_{n}^{\delta}(D)=S_{n}^{\varepsilon}(D)$, and therefore
$(\psi_{1},\psi_{2})$ is a skeletal non-critical interval.

\begin{remark}
Fix a field $\mathbb{F}$. We include the following terminology only to relate
skeletal homology to the persistence-module literature; it will not be used in
the remainder of the paper. The skeletal persistence module $\mathbb{K}%
_{n}(X;\mathbb{F}) = \bigl\{
K_{n}^{\varepsilon}(X;\mathbb{F}), j_{\varepsilon\delta} \bigr\}_{0<\delta
<\varepsilon} $ is called \emph{$q$-tame} if, for every $0<\delta<\varepsilon
$, the change-of-scale map $j_{\varepsilon\delta}: K_{n}^{\delta}%
(X;\mathbb{F}) \longrightarrow K_{n}^{\varepsilon}(X;\mathbb{F}) $ has
finite-dimensional image. Under this hypothesis, a persistence diagram can be
associated with $\mathbb{K}_{n}(X;\mathbb{F})$; see \cite{CDO}. Roughly
speaking, this diagram records homological features that persist over a
nonzero range of scales. However, $q$-tameness alone does not imply that the
skeletal critical spectrum $\kappa_{n}(X;\mathbb{F})$ is discrete, nor does it
imply that this spectrum can be recovered simply as the set of endpoints of
bars in the persistence diagram. Such a description requires additional
assumptions on the persistence module. Since we do not use persistence
diagrams or these additional assumptions below, we do not pursue this
discussion here.
\end{remark}

\begin{proposition}
\label{nci}Suppose that $X$ and $Y$ are $\delta$-homotopy equivalent, and
$(a-2\delta,b+2\delta)$ is a skeletal non-critical interval for $Y$. If
$\max\{a,2\delta\}<a^{\prime}<b^{\prime}<b$ then $(a^{\prime},b^{\prime})$ is
skeletal non-critical in $X$.
\end{proposition}

\begin{proof}
Using Diagram (\ref{HI}), we obtain the following diagram, in which all maps,
except \textit{a priori} $j_{b^{\prime}a^{\prime}}$, are isomorphisms.
\begin{equation}
\begin{tikzcd}[column sep=small, row sep=2.2em] & K_{n}^{a^{\prime}-2\delta}(Y) \arrow[d, "j_{a^{\prime},a^{\prime}-2\delta}"] \\ K_{n}^{a^{\prime}}(X) \arrow[ur, "g_{\ast}"] \arrow[d, "j_{b^{\prime},a^{\prime}}"'] & K_{n}^{a^{\prime}}(Y) \arrow[d, "j_{b^{\prime},a^{\prime}}"] \\ K_{n}^{b^{\prime}}(X) \arrow[dr, "f_{\ast}"'] & K_{n}^{b^{\prime}}(Y) \arrow[d, "j_{b^{\prime}+2\delta,b^{\prime}}"] \\ & K_{n}^{b^{\prime}+2\delta}(Y) \end{tikzcd} \label{CV}%
\end{equation}
But then $j_{b^{\prime}a^{\prime}}$ must also be an isomorphism. Since
$a^{\prime}$ and $b^{\prime}$ were arbitrary, we may also conclude that
$j_{b^{\prime\prime}a^{\prime\prime}}$ is an isomorphism for any subinterval
$(a^{\prime\prime},b^{\prime\prime})$ of $(a^{\prime},b^{\prime})$. Therefore,
$(a^{\prime},b^{\prime})$ is skeletal non-critical. We next give a very short
proof of the following known stability theorem (see Theorem 5.2 \cite{CDO}).
\end{proof}

\begin{theorem}
\label{spectrum}If $X$ and $Y$ are $\delta$-homotopy equivalent then
$d_{H}(\kappa_{n}(X),\kappa_{n}(Y))<2\delta$.
\end{theorem}

\begin{proof}
Suppose not. Then, without loss of generality, there is some $\psi\in
\kappa_{n}(X)$ such that there is no skeletal critical value of $Y$ in
$(\psi-2\delta^{\prime},\psi+2\delta^{\prime})$, for some $\delta^{\prime
}>\delta$. Suppose first that $\psi>0$. According to the above discussion,
there would have to be some open skeletal non-critical interval about $\psi$,
a contradiction. If $\psi=0$ then by definition there is some $\psi^{\prime
}>0$ in $\kappa_{n}(X)$ arbitrarily close to $\psi=0$, and the previous
argument applies.
\end{proof}

In Theorem 1.3 of \cite{MZ}, M\'{e}moli and Zhou proved an analogous stability
theorem for a persistent fundamental group, for compact, semi-locally simply
connected geodesic spaces. As we will explain after a bit more background in
the next section, for metric spaces with abelian fundamental groups, Theorem
\ref{spectrum}, $n=1$, generalizes their theorem by removing the additional
assumptions on the metric spaces.

Now suppose that $X_{i}\rightarrow Y$ in the Gromov-Hausdorff metric, where
$X_{i},Y$ are compact metric spaces. As was shown by Gromov, we may simply
assume that the compact metric spaces $X_{i},Y$ are subsets of a single
ambient metric space, and convergence is in the Hausdorff metric. The first
part of the next theorem is immediate from Theorem \ref{spectrum}, and the
second part is immediate from Diagram (\ref{CV}).

\begin{theorem}
\label{stab}Suppose $X_{i}\rightarrow Y$ in the Gromov-Hausdorff metric and
$n\geq1$.

\begin{enumerate}
\item The sets $\kappa_{n}(X_{i})$ are Cauchy in the Hausdorff metric.

\item If $\varepsilon>0$ is a skeletal non-critical value for $Y$, then for
all large enough $i$, $K_{n}^{\varepsilon}(X_{i})$ is isomorphic to
$K_{n}^{\varepsilon}(Y)$.
\end{enumerate}
\end{theorem}

According to Theorem \ref{stab}, the sets $\kappa_{n}(X_{i})$ converge in the
Hausdorff metric to a unique compact set $\kappa_{n}$, but it is possible that
$\kappa_{n}$ is not contained $\kappa_{n}(Y)$ (Example \ref{pac}). That is,
critical values may \textquotedblleft vanish\textquotedblright\ in the limit.

Recall that a subset $D$ of a metric space $X$ is $\mathit{\delta}%
$\textit{-dense for some }$\delta>0$, if for every $x\in X$ there is some
$d\in D$ such that $d(x,d)<\delta$. Under some circumstances, it is possible
to obtain information on the skeletal homology of $X$ from the skeletal
homology of $D$.

\begin{theorem}
\label{cycle}Suppose that $D$ is $\delta$-dense in a metric space $X$ and
$\varepsilon>0$. Then $i_{\ast}:K_{n}^{\varepsilon}(D)\rightarrow
K_{n}^{\varepsilon}(X)$ is injective (resp. surjective) if and only if
$j_{\varepsilon+2\delta,\varepsilon}:K_{n}^{\varepsilon}(D)\rightarrow
K_{n}^{\varepsilon+2\delta}(D)$ is injective (resp. $j_{\varepsilon
,\varepsilon-2\delta}:K_{n}^{\varepsilon-2\delta}(X)\rightarrow K_{n}%
^{\varepsilon}(X)$ is surjective).
\end{theorem}

\begin{proof}
Define a function $R^{D}:X\rightarrow D$ by letting $R^{D}(d)=d$ for any $d\in
D$, and for $x\in X$ letting $R^{D}(x)$ be any $d\in D$ such that
$d(x,d)<\delta$. That is, $R^{D}$ is a kind of \textquotedblleft metric
deformation retraction\textquotedblright. If $i:D\rightarrow X$ is the
inclusion and $i^{D}:D\rightarrow D$ is the identity, then as in singular
homology we may consider the compositions $R_{\ast}^{D}\circ i_{\ast}=i_{\ast
}^{D}$ and $i_{\ast}\circ R_{\ast}^{D}=R_{\ast}^{D}$, but the moduli of
continuity alter the conclusions. By the triangle inequality, $\omega
(\varepsilon)=\varepsilon-2\delta$ is a modulus of continuity for $R^{D}$, and
$i$ is $1$-Lipschitz. As in the proof of Proposition \ref{nci},
\[
R_{\ast}^{D}\circ i_{\ast}=j_{\varepsilon+2\delta,\varepsilon}:K_{n}%
^{\varepsilon}(D)\rightarrow K_{n}^{\varepsilon+2\delta}(D)
\]
The first statement now follows. The second part similarly follows,
considering the domain and range of
\[
i_{\ast}\circ R_{\ast}^{D}=j_{\varepsilon,\varepsilon-2\delta}:K^{\varepsilon
-2\delta}(X)\rightarrow K^{\varepsilon}(X)
\]

\end{proof}

\begin{theorem}
\label{dens}Let $X$ be a metric space and $\varepsilon>0$ be such that $X$ may
be covered by $N$ $\frac{\varepsilon}{4}$-balls. Then $H_{n}^{\varepsilon}(X)$
has a generating set with at most $N^{n+1}$ elements. In particular, if $X$ is
compact, then $H_{n}^{\varepsilon}(X)$ is finitely generated for all $n\geq1$
and $\varepsilon>0$.
\end{theorem}

\begin{proof}
Let $S=\{x_{1},\ldots,x_{N}\}$ be the centers of a cover of $X$ by $N$
$\frac{\varepsilon}{4}$-balls, and let $i:S\rightarrow X$ be the inclusion.
For every $\beta>0$, $S_{n}^{\beta}(S)$ is a free abelian group on at most
$N^{n+1}$ generators (there are only $N^{n+1}$ functions $V^{n}\rightarrow S$),
so its subgroup of cycles, and hence $K_{n}^{\beta}(S)$, is generated by at
most $N^{n+1}$ elements. We will show that the image of the inclusion-induced
map $i_{\ast}:K_{n}^{\frac{3\varepsilon}{4}}(S)\rightarrow K_{n}^{\varepsilon
}(X)$ contains $H_{n}^{\varepsilon}(X)$. Since a subgroup of an abelian group
generated by $m$ elements is itself generated by at most $m$ elements, this
will finish the proof.

Define $R:X\rightarrow S$ by letting $R(s)=s$ for $s\in S$, and for every
other $x\in X$ letting $R(x)$ be some $s\in S$ with $d(x,s)<\frac{\varepsilon
}{4}$. By the triangle inequality, $\omega(t)=\max\{t-\frac{\varepsilon}%
{2},0\}$ is a modulus of continuity for $i\circ R$, and it is also a modulus
of continuity for $\operatorname{id}_{X}$. Since $\left\vert i\circ
R-\operatorname{id}_{X}\right\vert <\frac{\varepsilon}{4}$ and $\omega
(\frac{3\varepsilon}{4})=\frac{\varepsilon}{4}>0$, Theorem \ref{induced}
gives
\[
(i\circ R)_{\ast}=(\operatorname{id}_{X})_{\ast}=j_{\varepsilon,\frac
{\varepsilon}{4}}:K_{n}^{\frac{\varepsilon}{4}}(X)\rightarrow K_{n}%
^{\varepsilon}(X)\text{.}%
\]
Now suppose that $[c]_{\varepsilon}\in H_{n}^{\varepsilon}(X)$. Since
$\eta_{\varepsilon,\frac{\varepsilon}{4}}$ is surjective, we may assume that
$c$ is an $\frac{\varepsilon}{4}$-cycle. Then $R_{\#}(c)$ is a
$\frac{3\varepsilon}{4}$-cycle in $S$, and
\[
\lbrack c]_{\varepsilon}=j_{\varepsilon,\frac{\varepsilon}{4}}([c]_{\frac
{\varepsilon}{4}})=(i\circ R)_{\ast}([c]_{\frac{\varepsilon}{4}}%
)=[R_{\#}(c)]_{\varepsilon}=i_{\ast}([R_{\#}(c)]_{\frac{3\varepsilon}{4}%
})\text{.}%
\]
\end{proof}

Gromov's precompactness criterion says that a collection $\mathcal{X}$ of
compact metric spaces is precompact (aka totally bounded) if for every
$\varepsilon>0$ there is a uniform bound $N(\varepsilon)$ on the number of
$\varepsilon$-balls required to cover any space in $\mathcal{X}$. Therefore,
we obtain:

\begin{corollary}
\label{fgen}If $\mathcal{X}$ is a Gromov-Hausdorff precompact collection of
compact metric spaces, then for every $n,\varepsilon>0$ there exists some
$H(\varepsilon,n)$ such that for every $X\in\mathcal{X}$, $H_{n}^{\varepsilon
}(X)$ is generated by $H(\varepsilon,n)$ elements.
\end{corollary}

We now turn our attention to Latschev's Theorem and its various generalizations.

\begin{definition}
For any metric space $X$ and $n$, we say that $X$ has positive $n^{th}%
$-stability radius $\sigma_{n}=\sigma_{n}(X)>0$ if for all $0<\varepsilon
<\sigma_{n}$, $\rho_{\varepsilon}:H_{n}(X)\rightarrow K_{n}^{\varepsilon}(X)$
is an isomorphism.
\end{definition}

It is immediate that if $X$ has stability radius $\sigma_{n}>0$ then
$\sigma_{n}$ is a lower bound for the skeletal critical spectrum $\kappa
_{n}(X)$. The conclusion of Theorem \ref{HM2} may now be restated as:
$\sigma_{n}(X)\geq2^{-(n+1)}\psi_{0}$.

\begin{theorem}
\label{qlatschev} Let $X$ and $Y$ be compact metric spaces such that $X$ and
$Y$ are $\delta$-homotopy equivalent and $X$ has stability radius $\sigma
_{n}>4\delta>0$. Then for all $0<\beta<\sigma_{n}-4\delta$,
\[
K_{n}^{\beta}(Y)\equiv K_{n}^{\beta}(X)\equiv H_{n}(X).
\]
Equivalently,
\[
H_{n}^{\Delta}\bigl(VR_{\beta}(Y)\bigr)\cong H_{n}(X)\text{.}%
\]
These assumptions are in particular satisfied if $d_{GH}(X,Y)<\delta$ and $X$
has uniform contractibility radius $\psi_{0}>0$, taking $\sigma_{n}%
=2^{-(n+1)}\psi_{0}$.
\end{theorem}

\begin{proof}
Under these assumptions, Proposition \ref{nci} implies that $(0,\sigma
_{n}-4\delta)$ is homotopy non-critical, and the homomorphism $h$ in Diagram
(\ref{HI}) is an isomorphism. The Vietoris--Rips statement follows from
Theorem~\ref{VR}, and the final statement follows from the comment prior to
the statement of the theorem.
\end{proof}

In Proposition 5.11, \cite{CDO}, Chazal-deSilva-Oudot give an example of a
totally bounded metric space $X$ such that the first simplicial homology of
the VR-complex of $X$ is countably generated, showing that Theorem \ref{dens}
fails for skeletal homology. As a totally disconnected space, their example
has trivial real homology and hence only phantom homology.

Clearly, it is important to know when the skeletal or homology critical
spectrum is discrete (in $(0,\infty)$). We say that $\kappa_{n}(X)$ is
discrete from below (resp. from above) if for every $\psi\in\kappa_{n}(X)$
there is some $\varepsilon>0$ such that $(\varepsilon,\psi)$ (resp.
$(\psi,\varepsilon)$) is non-critical with similar definitions for $\eta
_{n}(X)$. It is possible for a compact metric space to have a skeletal
critical spectrum that is neither discrete from above nor from below (Example
\ref{jd}). We can squeeze out some purely algebraic consequences from the
ascending chain condition, when the groups in question are finitely generated,
hence Noetherian.

\begin{proposition}
Suppose $\{\varepsilon_{i}\}$ is strictly increasing with $0<\beta
\leq\varepsilon_{i}\leq\psi$ for all $i$.
\end{proposition}

\begin{enumerate}
\item If\ $K_{n}^{\beta}(X)$ is finitely generated and for all large $i$,
$j_{\varepsilon_{i}\beta}$ is surjective, then for all large $i<j$,
$j_{\varepsilon_{j}\varepsilon_{i}}$ is injective

\item If $K_{n}^{\psi}(X)$ is finitely generated and for all large $i$,
$j_{\psi\varepsilon_{i}}$ is injective, then for all large $i<j$,
$j_{\varepsilon_{j}\varepsilon_{i}}$ is surjective.
\end{enumerate}

\begin{proof}
We have the following sequence, where $j>i$%
\[
K_{n}^{\beta}(X)\rightarrow\cdots\rightarrow K_{n}^{\varepsilon_{i}%
}(X)\rightarrow\cdots\rightarrow K_{n}^{\varepsilon_{j}}(X)\rightarrow
\cdots\rightarrow K_{n}^{\psi}(X)
\]
By the inverse limit condition, if $K_{i}$ denotes $\ker j_{\varepsilon
_{i}\beta}$, $K_{i}\subset K_{i+1}$; by definition, $\{K_{i}\}$ is an
ascending chain. Since $K_{n}^{\beta}(X)$ is finitely generated, by the
ascending chain condition, this sequence stabilizes. That is, there is some
$s$ such that for all $k>s$, $K_{s}=K_{k}$. Suppose that $[c]_{\varepsilon
_{s}}\in\ker j_{\varepsilon_{k}\varepsilon_{s}}$. If $j_{\varepsilon_{s}\beta
}$ is surjective, there is some $[c^{\prime}]_{\beta}\in K_{n}^{\beta}(X)$
such that $j_{\varepsilon_{s}\beta}([c^{\prime}]_{\beta})=[c]_{\varepsilon
_{s}}$, and $[c^{\prime}]_{\beta}\in K_{k}=K_{s}$. Therefore $[c]_{\varepsilon
_{s}}=0$, showing that $j_{\varepsilon_{k}\varepsilon_{s}}$ is injective and
proving the first statement.

For the second part, let $I_{i}$ be the image of $j_{\psi\varepsilon_{i}}$.
Again by the inverse limit condition, $\{I_{i}\}$ is an ascending chain in
$K_{n}^{\psi}(X)$, stabilizing at some $I_{s}$. Suppose that $k>s$ and let
$[c]_{k}\in K_{n}^{\varepsilon_{k}}(X)$. Then $j_{\psi\varepsilon_{k}%
}([c]_{\varepsilon_{k}})=[c]_{\psi}\in I_{k}=I_{s}$. Therefore, there is some
$[c^{\prime}]_{\varepsilon_{s}}\in K_{n}^{\varepsilon_{s}}(X)$ such that
$j_{\psi\varepsilon_{s}}([c^{\prime}]_{\varepsilon_{s}})=[c]_{\psi}$. But then
$[c]_{\varepsilon_{k}}-[c^{\prime}]_{\varepsilon_{k}}\in\ker j_{\psi
\varepsilon_{k}}$. If $\ker j_{\psi\varepsilon_{k}}=0$ then $[c]_{\varepsilon
_{k}}=[c^{\prime}]_{\varepsilon_{k}}$, showing that $j_{\varepsilon
_{k}\varepsilon_{s}}$ is surjective.
\end{proof}

Using similar algebraic arguments, since $\eta_{\varepsilon_{j}\varepsilon
_{i}}$ is always surjective, we have the following:

\begin{proposition}
\label{ch}If $H_{n}(X)$ is finitely generated, then the homology critical
spectrum of $X$ is discrete from below, and the number of non-discrete
homology critical values is bounded above by the rank of $H_{n}(X)$.
\end{proposition}

We now have a nice picture of how the groups $H_{n}^{\varepsilon}(X)$
\textquotedblleft stratify\textquotedblright\ the singular homology of a
metric space $X$ in a way that is much more orderly than skeletal homology or
equivalently the Vietoris-Rips persistence module. We have the commutative diagram%

\begin{equation}
\begin{tikzcd}[column sep=small, row sep=small] & {} \arrow[d] \\ & H_{n}^{\delta}(X) \arrow[d] \\ H_{n}(X) \arrow[ur] \arrow[r] \arrow[dr] & \vdots \arrow[d] \\ & H_{n}^{\varepsilon}(X) \arrow[d] \\ & {} \end{tikzcd} \label{rhdia}%
\end{equation}
in which all of the maps are surjective. The isomorphism type of
$H_{n}^{\varepsilon}(X)$ may only change at the homology critical values,
which are discrete from below if $H_{n}(X)$ is finitely generated. If $X$ has
stability radius $\sigma_{n}>0$ then the inverse system stabilizes at
$H_{n}(X)$ below $\sigma_{n}$. That is, as the scale goes from large to small
(upwards in the diagram), $H_{n}^{\varepsilon}(X)$ group and \textquotedblleft
sees more topology\textquotedblright\ until eventually it sees all of it.

\begin{remark}
Returning to the colorful analogy of \textquotedblleft birth\textquotedblright%
\ and \textquotedblleft death\textquotedblright\ of \textquotedblleft
features\textquotedblright\ as the parameter goes from small to large, we may
say the following for a compact, uniformly locally contractible space: Real
homology is \textquotedblleft born\textquotedblright\ for small enough
$\varepsilon$, \textquotedblleft dies a little\textquotedblright\ at each
homology critical value, and is \textquotedblleft fully dead\textquotedblright%
\ by the time the parameter exceeds the diameter of the space. On the other
hand, phantom homology starts out dead, then comes to life, then dies again.
\end{remark}

\begin{remark}
When it comes to Riemannian manifolds with non-negative curvature, with
diameter $\leq D$ and dimension $\leq n$, the uniform locally contractibility
requirement is much stronger than is necessary to uniformly bound the
generators of homology, as is seen by Gromov's Betti Numbers Theorem
(\cite{GB}) and its extension by Weiss in \cite{W}.
\end{remark}

Along these lines, it is interesting to note that, as is well-known from the
study of collapse of Riemannian manifolds, singular homology generally is
\textit{not stable }with respect to Gromov-Hausdorff convergence. Therefore,
through its stability, skeletal homology is able to \textit{anticipate} the
abrupt changes in homology that can occur in Gromov-Hausdorff limits. We will
look more closely at this phenomenon in \cite{SH2}, but for now we only
present the following example.

\begin{example}
\label{berg}The Berger spheres (\cite{Ber}) $B_{\delta}$ consist of Riemannian
metrics on $S^{3}$ that begin with the Fubini-Study metric and shrink the
fibers of the Hopf fibration $h:S^{3}\rightarrow S$, which are rescaled great
circles that remain orthogonal to the base space) $S^{2}$. We will not provide
any more geometric details about this important example because for our
purposes we only need the fact that the topological $3$-spheres $B_{\delta}$
converge in the Gromov-Hausdorff metric on the base $S$, which is a
topological $2$-sphere. Up to reparameterization, we have $\delta$-homotopy
equivalences $f_{\delta}:B_{\delta}\rightarrow S$ and $g_{\delta}:S\rightarrow
B_{\delta}$ for all $\delta>0$. As in the proof of Theorem \ref{interleaf},
for any $\varepsilon>4\delta$ the change of scale map $j_{\varepsilon
+2\delta,\varepsilon-4\delta}$ on $S$ factors through $K_{n}^{\varepsilon
-2\delta}(B_{\delta})$:
\[
K_{n}^{\varepsilon-4\delta}(S)\overset{(g_{\delta})_{\ast}}{\longrightarrow
}K_{n}^{\varepsilon-2\delta}(B_{\delta})\overset{(f_{\delta})_{\ast}%
}{\longrightarrow}K_{n}^{\varepsilon}(S)\overset{j_{\varepsilon+2\delta
,\varepsilon}}{\longrightarrow}K_{n}^{\varepsilon+2\delta}(S)\text{,}%
\]
and similarly $j_{\varepsilon+2\delta,\varepsilon-4\delta}$ on $B_{\delta}$
factors through $K_{n}^{\varepsilon-2\delta}(S)$. Since $S$ is a compact
Riemannian manifold, by Theorem \ref{HM2} it has positive stability radius
$\sigma_{n}(S)$ for every $n$; that is, $K_{n}^{t}(S)=H_{n}^{t}(S)\cong
H_{n}(S)$ for $0<t<\sigma_{n}(S)$, and the change of scale maps between these
groups are isomorphisms.

Consider first $n=2$. If $4\delta<\varepsilon$ and $\varepsilon+2\delta
<\sigma_{2}(S)$ then $j_{\varepsilon+2\delta,\varepsilon-4\delta}%
:K_{2}^{\varepsilon-4\delta}(S)\rightarrow K_{2}^{\varepsilon+2\delta}(S)$ is
an isomorphism $\mathbb{Z}\rightarrow\mathbb{Z}$, so $(g_{\delta})_{\ast}$ is
injective and $K_{2}^{\varepsilon-2\delta}(B_{\delta})$ contains a copy of
$\mathbb{Z}$. On the other hand, $H_{2}(B_{\delta})=H_{2}(S^{3})=0$, so
$H_{2}^{t}(B_{\delta})=0$ and $K_{2}^{t}(B_{\delta})=P_{2}^{t}(B_{\delta})$
for all $t>0$. Hence $P_{2}^{t}(B_{\delta})$ contains $\mathbb{Z}$ for all
$2\delta<t<\sigma_{2}(S)-4\delta$. Put another way, the phantom homology
$P_{2}^{t}(B_{\delta})$ \textquotedblleft anticipates\textquotedblright\ the
impending collapse.

As for $n=3$, $H_{3}(B_{\delta})=H_{3}(S^{3})=\mathbb{Z}$ while $H_{3}(S)=0$.
If $4\delta<\varepsilon$ and $\varepsilon-2\delta<\sigma_{3}(S)$, then
$K_{3}^{\varepsilon-2\delta}(S)=0$, so $j_{\varepsilon+2\delta
,\varepsilon-4\delta}:K_{3}^{\varepsilon-4\delta}(B_{\delta})\rightarrow
K_{3}^{\varepsilon+2\delta}(B_{\delta})$ is the zero map. Since $\eta
_{\varepsilon+2\delta,\varepsilon-4\delta}$ is the restriction of this map to
$H_{3}^{\varepsilon-4\delta}(B_{\delta})$ and is surjective, $H_{3}%
^{\varepsilon+2\delta}(B_{\delta})=0$. Using the surjectivity of the maps
$\eta$ once more, $H_{3}^{t}(B_{\delta})=0$ for all $t>6\delta$. On the other
hand, $B_{\delta}$ is a compact Riemannian manifold, so $H_{3}^{t}(B_{\delta
})=\mathbb{Z}$ for all $0<t<\sigma_{3}(B_{\delta})$. Therefore $\sigma
_{3}(B_{\delta})\leq6\delta$. In other words, although $H_{3}^{t}(B_{\delta
})=\mathbb{Z}$ for all small enough $t$, the upper bound of those values $t$
tends to $0$ with $\delta$. This time, the real homology $H_{3}^{t}%
(B_{\delta})$ anticipates the drop in dimension.
\end{example}

\textbf{Questions:}

\begin{enumerate}
\item What are the homology and skeletal critical spectra of spheres of
constant curvature $1$? From Example \ref{s1} we have $\kappa_{1}%
(S^{1})=\{\frac{1}{3}\}$, and (in the language of this paper) this question
was fully answered for $n=1$ by Adamaszek-Adams in \cite{AA}. For $n>1$ it is
easy to check by induction that $\kappa_{n}(S^{n})$ is bounded above by
$\frac{\pi}{2}$. The proof for $n=1$ follows from the existence of special
closed geodesics called \textit{essential circles} whose lengths determine the
homotopy critical spectrum, which is the same as $\kappa_{1}(S^{1})$. But we
know of no analog of essential circles for $n>1$. A subsequent step might be
to more precisely understand the homology critical spectrum of the Berger
spheres in Example \ref{berg}.

\item Do there exist compact, connected Riemannian manifolds, or even compact,
connected geodesic spaces, with $\kappa_{n}$ or $\eta_{n}$ non-discrete?
($\kappa_{1}=\eta_{1}$ is known to be discrete for compact geodesic spaces
(\cite{PW1}).)

\item Is it possible to characterize Riemannian $n$-manifolds that have no
phantom homology in dimensions $\leq n$?
\end{enumerate}

\section{Skeletal Homology}

In order to relate skeletal homology to classical singular homology, we need
to revisit one of the most important basic theorems in singular homology: that
homology is unchanged by simply restricting to \textquotedblleft small
simplices\textquotedblright, traditionally defined as simplices subordinate to
an open cover (Proposition 2.21 in \cite{H}). For any $\varepsilon>0$, let
$C_{S,n}^{\varepsilon}(X)$ denote the free abelian group generated by
\textquotedblleft$\varepsilon$-small\textquotedblright\ continuous simplices,
that is, whose images have a diameter less than $\varepsilon$. We denote the
resulting homology groups by $H_{S,n}^{\varepsilon}(X)$ to distinguish them
from $H_{n}^{\varepsilon}(X)$.

\begin{theorem}
\label{rh}Let $X$ be a metric space and $0<\delta<\varepsilon$. Then the
inclusion maps of $C_{S,n}^{\delta}(X)$ into $C_{n}(X)$, resp. $C_{S,n}%
^{\varepsilon}(X)$, induce isomorphisms $\beta_{\delta}$, resp. $\beta
_{\varepsilon\delta}$, such that the following diagram commutes:
\begin{equation}
\begin{tikzcd}[column sep=2.0em, row sep=2.0em] H_{S,n}^{\delta}(X) \arrow[r, "\beta_{\delta}"] \arrow[d, "\beta_{\varepsilon\delta}"'] & H_{n}(X) \\ H_{S,n}^{\varepsilon}(X) \arrow[ur, "\beta_{\varepsilon}"] & \end{tikzcd} \label{isos}%
\end{equation}

\end{theorem}

The proof is essentially the same as that of Proposition 2.21 in \cite{H}. The
only difference is that rather than using repeated barycentric subdivision of
a continuous simplex until the resulting simplices have a diameter smaller
than the Lebesgue number of the hypothesized open cover $\mathcal{U}$
restricted to the compact image of the simplex, one simply subdivides until
the images of the resulting simplices have diameter less than $\varepsilon$.

Suppose that $\sigma\in C_{S,n}^{\varepsilon}(X)$. Then the restriction of
$\sigma$ to $V^{n}$ is an element of $S_{n}^{\varepsilon}(X)$, and we may
define $r_{\varepsilon}:C_{S,n}^{\varepsilon}(X)\rightarrow S_{n}%
^{\varepsilon}(X)$ by linearly extending the restriction map. It is easy to
check that $r_{\varepsilon}$ is a chain map and we define $\rho_{\varepsilon
}:H_{n}(X)\rightarrow K_{n}^{\varepsilon}(X)$ as in the Introduction.

\begin{proof}
[Proof of Theorem \ref{cpd}]The fact that the vertical maps form an inverse
system is simply a restatement of the trivial fact that $j_{\varepsilon\delta
}\circ j_{\delta\tau}=j_{\varepsilon\tau}$ whenever $0<\tau<\delta
<\varepsilon$, with similar statements for the other vertical maps. The
exactness of the horizontal sequences is standard, with
\begin{equation}
\phi_{\varepsilon\delta}(\pi_{\delta}([c]_{\delta})):=\pi_{\varepsilon
}(j_{\varepsilon\delta}([c]_{\delta}))=\pi_{\varepsilon}([c]_{\varepsilon})
\label{alg}%
\end{equation}

If $[c]_{\delta}\in H_{n}^{\delta}(X)$ then%
\[
i_{\varepsilon}(\eta_{\varepsilon\delta}([c]_{\delta}))=i_{\varepsilon
}([c]_{\varepsilon})=[c]_{\varepsilon}=j_{\varepsilon\delta}([c]_{\delta
})=j_{\varepsilon\delta}(i_{\delta}([c]_{\delta}))
\]
This shows the commutativity of the left square. If $[c]_{\delta}\in
K_{n}^{\delta}(X)$ then by (\ref{alg}),
\[
\pi_{\varepsilon}(j_{\varepsilon_{\delta}}([c]_{\delta}))=\pi_{\varepsilon
}([c]_{\varepsilon})=\phi_{\varepsilon\delta}(\pi_{\delta}([c]_{\delta}))
\]
finishing the proof of commutativity.
\end{proof}

\begin{remark}
\label{facts}There are some additional facts that may be useful.

\begin{enumerate}
\item If $j_{\varepsilon\delta}$ is injective then $\eta_{\varepsilon\delta}$
is an isomorphism and $\phi_{\varepsilon\delta}$ is injective. In fact, it is
a quick diagram chase to see that $\ker\eta_{\varepsilon\delta}\subset\ker
j_{\varepsilon\delta}$ and since $\eta_{\varepsilon\delta}$ is surjective, the
injectivity of $j_{\varepsilon\delta}$ implies that $\eta_{\varepsilon\delta}$
is an isomorphism. The statement about $\phi_{\varepsilon\delta}$ is also a
simple diagram chase.

\item If $j_{\varepsilon\delta}$ is surjective then $\phi_{\varepsilon\delta}$
is surjective. Also, a diagram chase. Combining with the second fact yields
that if $j_{\varepsilon\delta}$ is an isomorphism then both $\eta
_{\varepsilon\delta}$ and $\phi_{\varepsilon\delta}$ are isomorphisms.
\end{enumerate}
\end{remark}

Discrete homotopy theory (\cite{BPTG}, \cite{BPUU}, \cite{PW1}, \cite{PW2})
provides many useful examples due to the relationship between the groups
$K_{1}^{\varepsilon}(X)$ and $\pi_{\varepsilon}(X,\ast)$, given by Theorem
\ref{fh2} below. We will quickly review a little background, almost entirely
adapted from \cite{PW1}. For any metric space $X$ and $\varepsilon>0$, an
$\varepsilon$-chain is a finite sequence $\{x_{0},...,x_{n}\}$ such that for
all $i$, $d(x_{i},x_{i+1})<\varepsilon$. There are two \textit{basic moves}:
adding a point between some $x_{i}$ and $x_{i+1}$, and removing a point
(except for an endpoint). A basic move is permissible if the resulting chain
is still an $\varepsilon$-chain. A finite sequence of basic moves is called an
$\varepsilon$-homotopy between $\varepsilon$-chains. The basic idea is that
while a (continuous) homotopy \textquotedblleft sees all 1-dimensional
holes\textquotedblright, an $\varepsilon$-homotopy only \textquotedblleft sees
holes at the scale of $\varepsilon$\textquotedblright. In a compact geodesic
space, these \textquotedblleft holes\textquotedblright\ are realized as
special closed geodesics of length $3\varepsilon$, called \textquotedblleft
essential circles\textquotedblright. Fixing a basepoint $\ast$, one may mimic
the construction of the fundamental group (and the universal cover, for that
matter, but we won't use that construction here). Concretely, given a
basepoint $\ast\in X$, $\pi_{\varepsilon}(X,\ast)$ is the group of
$\varepsilon$-homotopy equivalence classes $[\lambda]_{\varepsilon}$ of
$\varepsilon$-loops at $\ast$. The group operation is induced by
concatenation. If $0<\delta<\varepsilon$, every $\delta$-loop is an
$\varepsilon$-loop and there is a well-defined \textquotedblleft change of
scale\textquotedblright\ homomorphism $\theta_{\varepsilon\delta}:\pi_{\delta
}(X,\ast)\rightarrow\pi_{\varepsilon}(X,\ast)$ defined by $\theta
_{\varepsilon\delta}([\lambda]_{\delta})=[\lambda]_{\varepsilon}$. These maps
trivially form an inverse system. For geodesic spaces, these maps are
surjective. If $\varepsilon$ is larger than the diameter of $X$, then every
basic move is permissible, and any loop may be $\varepsilon$-homotoped to the
trivial loop--that is, $\pi_{\varepsilon}(X,\ast)=1$. If $X$ is compact and
semilocally simply connected then for some $\mu_{0}>0$, $\pi_{\varepsilon
}(X,\ast)$ is isomorphic to $\pi_{1}(X)$ for all $0<\varepsilon<\mu_{0}$. That
is, as $\varepsilon\rightarrow0$, $\pi_{\varepsilon}(X,\ast)$
\textquotedblleft sees more and more of the fundamental
group\textquotedblright, and is isomorphic to the fundamental group for small
enough $\varepsilon>0$. A positive number $\psi$ is homotopy
\textit{non-critical} if there exist $0<\delta<\psi<\varepsilon$ so that
$\theta_{\varepsilon\delta}$ is an isomorphism. Otherwise, $\psi$ is called
\textit{homotopy critical}, and the set of all homotopy critical values is
called the \textit{homotopy critical spectrum}. If $X$ is a compact geodesic
space then the set of homotopy critical values is discrete in $(0,\infty)$. In
other words, there is a discrete set of values at which the isomorphism type
of $\pi_{\varepsilon}(X)$ changes, and this discreteness is very important to
prove stability about the groups $\pi_{\varepsilon}(X)$. Stability, in turn,
is a consequence of a result of basic importance that roughly says that
\textquotedblleft$\delta$-close $\varepsilon$-chains are $(\varepsilon
+\delta)$-homotopic\textquotedblright\ (see \cite{PW1} for an exact
statement). In \cite{PW1}, Plaut-Wilkins used discrete homotopy theory to
generalize fundamental group finiteness theorems of Anderson (\cite{An}) and
Shen-Wei (\cite{ShW}) for compact Riemannian manifolds--as special cases of a
finiteness theorem for compact geodesic spaces. We note that, up to
isomorphism, the groups $\pi_{\varepsilon}(X,\ast)$ for geodesic spaces were
independently discovered by Sormani-Wei (\cite{SW1}) through an entirely
different construction that did not use discrete methods.

For the next theorem, a metric space $X$ is called $\varepsilon$-connected if
every pair of points in $X$ is joined by an $\varepsilon$-chain. The term
\textquotedblleft$\varepsilon$-component\textquotedblright\ of a point $p$
mentioned earlier is the largest $\varepsilon$-connected set containing $p$.

\begin{theorem}
(\cite{DD})\label{fh2}Let $X$ be an $\varepsilon$-connected metric space with
basepoint $\ast$. Then there is a natural surjective homomorphism
$h_{\varepsilon}:\pi_{\varepsilon}(X,\ast)\rightarrow K_{1}^{\varepsilon}(X)$,
the kernel of which is the commutator subgroup of $\pi_{\varepsilon}(X,\ast)$.
Moreover, one has the following commutative diagram for any $0<\delta
<\varepsilon$:%
\begin{equation}
\begin{tikzcd}[column sep=2.0em, row sep=2.0em] \pi_{\delta}(X,\ast) \arrow[r, "h"] \arrow[d, "\theta_{\varepsilon\delta}"'] & K_{1}^{\delta}(X) \arrow[d, "j_{\varepsilon\delta}"] \\ \pi_{\varepsilon}(X,\ast) \arrow[r, "h"] & K_{1}^{\varepsilon}(X) \end{tikzcd} \label{diac}%
\end{equation}

\end{theorem}

As previously mentioned, inverse systems of such as $\{\pi_{\varepsilon
}(X,\ast),\theta_{\varepsilon\delta}\}$ may be equivalently viewed as
persistence modules, and therefore the family of homomorphisms
$\{h_{\varepsilon}\}$ comprise a morphism of persistence modules from the
fundamental group persistence module to the first homology persistence module.
In \cite{MZ}, M\'{e}moli and Zhou showed that the fundamental group
persistence module may equivalently be described in two other ways. As
mentioned in the introduction, they also proved, for compact, semi-locally
simply connected geodesic spaces, precisely the same stability theorem as
Theorem \ref{stab}, $n=1$. Theorem \ref{stab} and Theorem \ref{fh2} show that
the stability theorem of \cite{MZ} is true for arbitrary metric spaces with
abelian fundamental groups, but we do not know whether this is true in general.

It is useful to recall some results and examples from discrete homotopy
theory, which translate into examples about the first skeletal homology group
via Theorem \ref{fh2}. We will suppress the basepoint notation in these
examples since for a path-connected space, the isomorphism class of
$\pi_{\varepsilon}(X,\ast)$ is independent of basepoint choice in an analogous
way to the classical result (\cite{BPUU}). According to Theorem \ref{fh2},
when $\pi_{\varepsilon}(X,\ast)$ is abelian, the homotopy critical values of
$X$ are precisely the same as the first skeletal homology critical values of
$X$.

\begin{example}
\label{s1}Let $C$ be the geodesic circle of circumference $1$. From \cite{PW1}
we know that $C$ has exactly one homotopy critical value, hence (by Theorem
\ref{fh2}) one $1$-homology critical value, namely $\frac{1}{3}$. That is,
$K_{1}^{\varepsilon}(C)$ is trivial for $\varepsilon\geq\frac{1}{3}$ and is
$K_{1}^{\varepsilon}(C)=\mathbb{Z}=H_{1}(C)$ for $\varepsilon<\frac{1}{3}$.
Since the maps $\theta_{\varepsilon\delta}$ are surjective for geodesic
spaces, (\ref{diac}) shows that the maps $j_{\varepsilon\delta}$ are
surjective for all $0<\delta<\varepsilon$. For $n>1$, Adamaszek-Adams
(\cite{AA}) showed, in our language, that the circle has phantom homology in
all dimensions above $1$, and determined exactly what we call the skeletal
critical spectrum in all dimensions.
\end{example}

\begin{example}
\label{pac}Remove an open segment of length $0<\delta<\frac{1}{3}$ from the
geodesic circle $C$ to obtain a metric space $P$. Here we take the subspace
metric, not the induced geodesic metric, which would simply be an interval. In
this case, there are two homotopy critical values, hence two $1$-skeletal
homology critical values: $\psi_{1}=\frac{1}{3}$ and $\psi_{2}=\delta$. The
idea is that $\pi_{\varepsilon}(P)=K_{1}^{\varepsilon}(P)$ \textquotedblleft
sees\textquotedblright\ a \textquotedblleft phantom hole\textquotedblright\ as
soon as $\varepsilon$-chains can \textquotedblleft cross the
gap\textquotedblright\ $(\varepsilon=\delta)$ and until $\pi_{\varepsilon}(P)$
no longer \textquotedblleft sees the hole\textquotedblright\ at $\varepsilon
=\frac{1}{3}$). That is, $\pi_{\varepsilon}(P)=K_{1}^{\varepsilon}(P)$
transitions from trivial, to $\mathbb{Z}$, and back to trivial at those two
values of $\varepsilon$. In our current terminology, if $\varepsilon<\frac
{1}{3}$ then $j_{\frac{1}{3},\varepsilon}:\mathbb{Z\rightarrow}0$ is not
injective. When $0<\varepsilon<\frac{1}{3}$, $j_{\varepsilon,\frac{1}{3}%
}:0\rightarrow\mathbb{Z}$ is not surjective. Now take the spaces $P_{i}$ with
gap length $\delta_{i}\nearrow\frac{1}{3}$, and let $P_{0}$ be the Hausdorff
limit of these spaces, which has gap length equal to $\frac{1}{3}$. The
homotopy critical spectrum of $P_{i}$ is $\{\delta_{i},\frac{1}{3}\}$, which
converges in the Hausdorff metric to $\{\frac{1}{3}\}$. But no $\frac{1}{3}%
$-chain can cross the gap in $P_{0}$, and its homotopy critical spectrum is
$0$. This shows that if $X_{i}\rightarrow Y$ in the Gromov-Hausdorff metric
and $\kappa_{i}$ is a sequence of skeletal or homology critical values
$\kappa_{i}$ of $X_{i}$ with $\kappa_{i}\rightarrow\kappa>0$, $\kappa$ may not
be a homology or skeletal critical value of $Y$. This is analogous to the
\textquotedblleft disappearance\textquotedblright\ of persistence bars in a
persistence diagram.
\end{example}

\begin{example}
\label{jd}Conant, Curnutte, Jones, Plaut, Pueschel and Wilkins (\cite{CP})
constructed a class of examples called Rapunzel's combs as part of an REU
project. Rather than making a gap in a geodesic circle, as in Example
\ref{pac}, one puts a gap in each of the long sides (of fixed length $1$) of a
relatively narrow rectangle in the plane with the subspace metric. By varying
the widths of the shorter sides and lengths of the gaps and rescaling, one can
create any pair of homotopy critical values. These individual pieces can be
\textquotedblleft stacked\textquotedblright\ to obtain Rapunzel's combs that
have a critical spectrum that is not discrete--with positive limit points of
both increasing and decreasing sequences. Note that this space is not locally
path connected, and has trivial singular homology, hence only phantom
homology. Wilkins (\cite{WD}) took infinitely many Rapunzel's combs and glued
them together to get a bounded (but not compact) subset of separable Hilbert
space that has homotopy critical values equal to the interval $(0,1]$.
\end{example}

\begin{example}
\label{rev}Consider the surface of revolution $M$ of the graph of $x=e^{z}$
about the $x$-axis, with the induced Riemannian metric (or even the subspace
metric, if one chooses). As shown in \cite{PW1}, $\pi_{\varepsilon}(M,\ast)=1$
for all $\varepsilon>0$. The idea is that any loop at the basepoint $\ast$
that generates $\pi_{1}(M,\ast)$, i.e. it \textquotedblleft wraps
around\textquotedblright\ this topological cylinder one time, is
fixed-endpoint homotopic to a loop that consists of a path $\alpha$ from
$\ast$ to a point $(-t,0,e^{-t})$, concatenated by a loop $\lambda$ whose
image is the intersection of $M$ with the plane $x=-t$, concatenated with the
reversal $\overline{\alpha}$ of $\alpha$. For fixed $\varepsilon>0$ and large
enough $t$, the loop $\lambda$ is $\varepsilon$-null because the
\textquotedblleft hole is small enough to cross with an $\varepsilon
$-homotopy\textquotedblright. By Theorem \ref{fh2}, $K_{1}^{\varepsilon}(M)=0$
for all $\varepsilon>0$. Viewed directly as homology, for any fixed
$\varepsilon>0$, any cycle that generates continuous $1$-homology bounds a
$(2,\varepsilon)$-chain that includes a skeletal $1$-simplex of diameter less
than $\varepsilon$ that \textquotedblleft wraps around\textquotedblright\ the
cylinder. This example shows that for non-compact spaces, $H_{n}^{\varepsilon
}(X)$ may \textquotedblleft miss\textquotedblright\ some continuous homology
for all $\varepsilon$.
\end{example}

\begin{example}
\label{2t} Take flat $2$-tori $T_{i}$ obtained from rectangles with side
lengths $1$ and $\frac{1}{i}$. As was shown in \cite{PW1}, $T_{i}$ has
homotopy critical spectrum $\{\frac{i}{3},\frac{1}{3}\}$, which converges to
$\{0,\frac{1}{3}\}$ in the Hausdorff metric. The Gromov-Hausdorff limit of
these tori is the geodesic circle $C$ mentioned above, which has homotopy
critical spectrum $\{\frac{1}{3}\}$, meaning that one homotopy/homology
critical value vanishes in the end.
\end{example}

As mentioned previously, combinatorial arguments from singular homology carry
over readily to skeletal homology. To illustrate this, we consider homology of
pairs, using Hatcher, Section 2.1, as a notational model. The necessary
arguments are for the most part exactly the same as those for classical
singular homology, and we will omit many details. For $A\subset X$, define
$S_{n}^{\varepsilon}(X,A)$ to be the quotient $S_{n}^{\varepsilon}%
(X)/S_{n}^{\varepsilon}(A)$. We have an induced boundary map $\partial
_{\varepsilon}:S_{n}^{\varepsilon}(X,A)\rightarrow S_{n-1}^{\varepsilon}(X,
A)$, and we let $K_{n}^{\varepsilon}(X,A)$ denote the homology of the
resulting chain complex. We have the following commutative diagram:%
\[
\begin{tikzcd}[column sep=2.0em, row sep=2.0em]
0
\arrow[r]
& S_{n}^{\varepsilon}(A)
\arrow[r, "i"]
\arrow[d, "\partial"']
& S_{n}^{\varepsilon}(X)
\arrow[r, "\pi"]
\arrow[d, "\partial"']
& S_{n}^{\varepsilon}(X,A)
\arrow[r]
\arrow[d, "\partial"]
& 0
\\
0
\arrow[r]
& S_{n-1}^{\varepsilon}(A)
\arrow[r, "i"']
& S_{n-1}^{\varepsilon}(X)
\arrow[r, "\pi"']
& S_{n-1}^{\varepsilon}(X,A)
\arrow[r]
& 0
\end{tikzcd}
\]

From this point, it is a purely algebraic fact that there is a long exact
sequence
\[
\cdots\longrightarrow K_{n}^{\varepsilon}(A) \xrightarrow{i_{\ast}}
K_{n}^{\varepsilon}(X) \xrightarrow{\pi_{\ast}} K_{n}^{\varepsilon}(X,A)
\xrightarrow{\partial} K_{n-1}^{\varepsilon}(A) \longrightarrow\cdots.
\]
where the boundary map is defined by taking the homology class of the
equivalence class of a chain $b\in K_{n}^{\varepsilon}(X)$ to
$[a]_{\varepsilon}\in K_{n-1}^{\varepsilon}(A)$, where $a$ is such that
$i(a)=\partial b$.

\section{VR Complexes}

Given a metric space $X$, recall that the Vietoris-Rips (VR) complex
$VR_{\varepsilon}(X)$ at the scale of $\varepsilon>0$ consists of the
simplicial complex, the $n$-simplices $\Sigma=\{x_{0},...,x_{n}\}$ of which
are subsets of $X$ of diameter less than $\varepsilon$ with $n+1$ points. The
main result of this section is that $K_{\varepsilon}^{n}(X)$ is isomorphic to
the simplicial homology $H_{n}^{\Delta}(VR_{\varepsilon}(X))$. The model
theorem is the equivalence of simplicial and singular homology (Theorem 2.27
in \cite{H}). However, that proof relies on the simplicial homology
$H_{n}^{\Delta}(X^{k},X^{k-1})$ of pairs of the $k$-skeleta of the complex
$X$. For our theorem, $X$ is simply a metric space, with no notion of
$k$-skeleta. Therefore, a new approach is needed. We exploit the fact that the
domains of skeletal simplices are finite, and replace the long exact sequence
for the pair $(X^{k}, X^{k-1})$ by a long exact sequence involving the
cardinality of the images of $n$-simplices.

\begin{definition}
\label{fakepair}Let $X$ be a metric space, $k,n$ be natural numbers. We define
$S_{n,k}^{\varepsilon}(X)\subset S_{n}^{\varepsilon}(X)$ to be the free
abelian group generated by all simplices $\sigma$ such that the cardinality
$\left\vert \sigma(V^{n})\right\vert $ is at most $k+1$. In order to maintain
a formal analogy with the traditional sequence of pairs, we denote
$S_{n,k}^{\varepsilon}(X)/S_{n,k-1}^{\varepsilon}(X)$ by $S_{n}^{\varepsilon
}(X^{k},X^{k-1})$.
\end{definition}

For use below, we note that $S_{n}^{\varepsilon}(X^{k},X^{k-1})$ is the free
abelian group generated by equivalence classes of $(n,\varepsilon)$-simplices
$\sigma$ such that $\left\vert \sigma(V^{n})\right\vert =k+1$. We have a
sequence of inclusions
\[
S_{n,0}^{\varepsilon}(X)\subset S_{n,1}^{\varepsilon}(X)\subset\cdots\subset
S_{n,n}^{\varepsilon}(X)=S_{n}^{\varepsilon}(X)
\]
Note that if $\left\vert \sigma(V^{n})\right\vert \leq k+1$ then $\left\vert
\partial\sigma(V^{n})\right\vert \leq k+1$ since $\partial\sigma$ is a linear
combination of restrictions of $\sigma$. Moreover, since the argument that
$\partial^{2}=0$ is purely combinatoric, for each $k$ we have a chain complex
\[
\cdots\xrightarrow{\partial}S_{n+1,k}^{\varepsilon}%
(X)\xrightarrow{\partial}S_{n,k}^{\varepsilon}(X)\xrightarrow{\partial}\cdots
\xrightarrow{\partial}S_{k,k}^{\varepsilon}(X)\xrightarrow{\partial}\cdots
\xrightarrow{\partial}S_{0,k}^{\varepsilon}(X)\longrightarrow0
\]
We will denote the homology of this sequence by $K_{n,k}^{\varepsilon}(X)$.
For $k\geq n$, $S_{n,k}^{\varepsilon}(X)=S_{n}^{\varepsilon}(X)$ and therefore
$K_{n,k}^{\varepsilon}(X)=K_{n}^{\varepsilon}(X)$. Even so, the formal process
below will be carried out without assuming $k\leq n$. We have the inclusion
$i:S_{n,k-1}^{\varepsilon}(X)\rightarrow S_{n,k}^{\varepsilon}(X)$ and the
quotient map $j:S_{n,k}^{\varepsilon}(X)\rightarrow S_{n,k}^{\varepsilon
}(X)/S_{n,k-1}^{n}(X)$.

We have the following commutative diagram:%
\begin{equation}
\begin{tikzcd}[column sep=2.0em, row sep=2.0em] 0 \arrow[r] & S_{n,k-1}^{\varepsilon}(X) \arrow[r, "i"] \arrow[d, "\partial"'] & S_{n,k}^{\varepsilon}(X) \arrow[r, "j"] \arrow[d, "\partial"'] & S_{n}^{\varepsilon}(X^{k},X^{k-1}) \arrow[r] \arrow[d, "\partial"] & 0 \\ 0 \arrow[r] & S_{n-1,k-1}^{\varepsilon}(X) \arrow[r, "i"'] & S_{n-1,k}^{\varepsilon}(X) \arrow[r, "j"'] & S_{n-1}^{\varepsilon}(X^{k},X^{k-1}) \arrow[r] & 0 \end{tikzcd} \label{sk3}%
\end{equation}

The right-hand boundary map is induced on the quotient. We define the groups
$K_{n}^{\varepsilon}(X^{k}, X^{k-1})$ to be the homology groups of the
right-hand vertical chain complex. Whenever one has a commutative diagram of
the type (\ref{sk3}), there is a formal algebraic process to obtain a long
exact sequence, see \cite{H}, p. 116. For those less familiar with this
process, we provide a couple of details. We define the boundary map
$\partial:K_{n}^{\varepsilon}(X^{k},X^{k-1})\rightarrow K_{n-1,k-1}%
^{\varepsilon}(X)$ as follows. Every element of $K_{n}^{\varepsilon}%
(X^{k},X^{k-1})$ is the homology class $[\overline{c}]$ of a \textit{relative
cycle }$\overline{c}$; that is, $c\in S_{n,k}^{\varepsilon}(X)$ such that
$\partial c\in S_{n-1,k-1}^{\varepsilon}(X)$. Since $\partial^{2}=0$,
$\partial c$ is a cycle in $S_{n-1,k-1}^{\varepsilon}(X)$. Therefore we may
define $\partial\lbrack\overline{c}]=[\partial c]\in K_{n-1,k-1}^{\varepsilon
}(X)$. Some diagram chasing establishes the following long exact sequence for
any fixed $k\geq0$:
\begin{equation}
\cdots\xrightarrow{\partial} K_{n,k-1}^{\varepsilon}(X) \xrightarrow{i_{\ast}}
K_{n,k}^{\varepsilon}(X) \xrightarrow{j_{\ast}} K_{n}^{\varepsilon}%
(X^{k},X^{k-1}) \xrightarrow{\partial} K_{n-1,k-1}^{\varepsilon}(X)
\xrightarrow{i_{\ast}} \cdots\label{2rr}%
\end{equation}

To be clear, \textquotedblleft$X^{k}$\textquotedblright\ by itself has no
meaning because the metric space $X$ may not have a $k$-skeleton. Therefore,
this sequence should not be confused with the long exact sequence for the
pair. When $n<k$, the groups $K_{n}^{\varepsilon}(X^{k},X^{k-1})$ vanish
because $K_{n,k}^{\varepsilon}(X)=K_{n}^{\varepsilon}(X)$ and $K_{n-1,k-1}%
^{\varepsilon}(X)=K_{n-1}^{\varepsilon}(X)$.

\begin{theorem}
\label{30}Let $X$ be a metric space and let $\varepsilon>0$. Then
\[
K_{n,k}^{\varepsilon}(X)=0\qquad\text{and}\qquad K_{n}^{\varepsilon}%
(X^{k},X^{k-1})=0
\]
whenever $0\leq k<n$.
\end{theorem}

\begin{proof}
We prove the first statement by induction on $k$. For $k=0$, $S_{n,0}%
^{\varepsilon}(X)$ is generated by constant skeletal $n$-simplices. We denote
the constant $m$-simplex with $\sigma(V^{m})=x$ by $\sigma_{x}^{m}$. Then
\[
\partial\sigma_{x}^{m}=\sum_{j=0}^{m}(-1)^{j}\sigma_{x}^{m-1}=%
\begin{cases}
0 & m\text{ odd,}\\
\sigma_{x}^{m-1} & m\text{ even.}%
\end{cases}
\]
If $n$ is odd, then every $\sigma_{x}^{n}$ is a cycle and $\partial\sigma
_{x}^{n+1}=\sigma_{x}^{n}$, showing that $K_{n,0}^{\varepsilon}(X)=0$.
\newline If $n$ is even and $c=\sum_{i}k_{i}\sigma_{x_{i}}^{n}$ is a cycle,
then
\[
\partial\left(  \sum_{i}k_{i}\sigma_{x_{i}}^{n+1}\right)  =\sum_{i}%
k_{i}\partial\sigma_{x_{i}}^{n+1}=\sum_{i}k_{i}\sigma_{x_{i}}^{n}=c.
\]
Thus $K_{n,0}^{\varepsilon}(X)=0$ in this case as well.

Now suppose $m\geq1$ and assume inductively that $K^{\varepsilon}_{r,m-1}(X)=0
$ whenever $r>m-1$. We prove that $K^{\varepsilon}_{n,m}(X)=0 $ whenever $n>m$.

Consider the following portion of (\ref{2rr}):%
\[
\cdots\xrightarrow{\partial} K_{n,m-1}^{\varepsilon}(X) \xrightarrow{i_{\ast}}
K_{n,m}^{\varepsilon}(X) \xrightarrow{j_{\ast}} K_{n}^{\varepsilon}%
(X^{m},X^{m-1}) \xrightarrow{\partial} K_{n-1,m-1}^{\varepsilon}(X)
\xrightarrow{i_{\ast}} \cdots
\]
Since $n>m$, we have $n>m-1\text{ and }n-1>m-1$. Therefore, by the inductive
hypothesis, $K_{n,m-1}^{\varepsilon}(X)=0\text{ and } K_{n-1,m-1}%
^{\varepsilon}(X)=0$. Hence $j_{\ast}:K_{n,m}^{\varepsilon}(X)\rightarrow
K_{n}^{\varepsilon}(X^{m},X^{m-1})$ is an isomorphism. It remains to show that
$K_{n}^{\varepsilon}(X^{m},X^{m-1})=0$ for $n>m$.

Let $[c]\in K_{n}^{\varepsilon}(X^{m},X^{m-1})$. By the description of the
quotient group, we may represent $c$ by a relative cycle of the form
$c=\sum_{i}k_{i}\sigma_{i}$, where each $\sigma_{i}$ is an $(n,\varepsilon
)$-simplex satisfying $|\sigma_{i}(V^{n})|=m+1$. For a generator $\sigma$,
write $\sigma(v_{i})=x_{i}$. Define $T(\sigma)$ to be the skeletal
$(n+1)$-simplex given by $T(\sigma)(v_{0})=x_{0},T(\sigma)(v_{i})=x_{i-1}%
\quad\text{for }i>0$. Then $T(\sigma)\in S_{n+1,m}^{\varepsilon}(X)$, since it
has the same image as $\sigma$. Extend $T$ linearly to chains. A direct
calculation gives $\partial T(\sigma)=\sigma-T(\partial\sigma)$.\newline Thus,
if $b=\sum_{i}k_{i}T(\sigma_{i})$, then $\partial b=c-T(\partial c)$. Since
$c$ is a relative cycle, $\partial c\in S_{n-1,m-1}^{\varepsilon}(X)$. Because
$T$ preserves image cardinality, $T(\partial c)\in S_{n,m-1}^{\varepsilon}%
(X)$. Therefore, $T(\partial c)$ vanishes in the quotient $S_{n}^{\varepsilon
}(X^{m},X^{m-1})=S_{n,m}^{\varepsilon}(X)/S_{n,m-1}^{\varepsilon}(X)$. Hence,
in the quotient complex, $c=\partial b$. Thus $c$ is a relative boundary, and
so $K_{n}^{\varepsilon}(X^{m},X^{m-1})=0$.

Since $\pi_{*}$ is an isomorphism, it follows that $K^{\varepsilon}_{n,m}(X)=0
$ for all $n>m$. This completes the induction.
\end{proof}

\begin{proof}
[Proof of Theorem \ref{VR}]Let $VR_{\varepsilon}^{k}(X)$ denote the
$k$-skeleton of the Vietoris--Rips complex $VR_{\varepsilon}(X)$. We compare
the filtration by skeleta $VR_{\varepsilon}^{0}(X)\subseteq VR_{\varepsilon
}^{1}(X)\subseteq\cdots\subseteq VR_{\varepsilon}^{k}(X)\subseteq\cdots$ with
the image-cardinality filtration of the skeletal chain complex. More
precisely, for every $n,k\geq0$, let $S_{n,k}^{\varepsilon}(X)\subseteq
S_{n}^{\varepsilon}(X)$ be the subgroup generated by the $(n,\varepsilon
)$-skeletal simplices whose images contain at most $k+1$ points. Thus, for
each fixed $n$, $S_{n,0}^{\varepsilon}(X)\subseteq S_{n,1}^{\varepsilon
}(X)\subseteq\cdots\subseteq S_{n,k}^{\varepsilon}(X)\subseteq\cdots\subseteq
S_{n}^{\varepsilon}(X).$ For each $k\geq0$, define a chain map
\[
\kappa_{n}^{k}:C_{n}^{\Delta}\bigl(VR_{\varepsilon}^{k}%
(X)\bigr)\longrightarrow S_{n,k}^{\varepsilon}(X)
\]
on an oriented simplex by $[x_{0},\ldots,x_{n}]\longmapsto\sigma$, where
$\sigma(v_{i})=x_{i}$.

The image of $\sigma$ has cardinality $n+1\leq k+1$, so $\sigma\in
S_{n,k}^{\varepsilon}(X)$. Since the boundary maps on both sides are given by
the same alternating sum over faces, $\kappa_{n}^{k}$ is a chain map.
Therefore, it induces homomorphisms
\[
\kappa_{\ast,n}^{k}:H_{n}^{\Delta}(VR_{\varepsilon}^{k}(X))\longrightarrow
K_{n,k}^{\varepsilon}(X)\text{.}%
\]
We prove by induction on $k$ that $\kappa_{\ast, n}^{k}$ is an isomorphism for
every $k$.

For $k=0$, $VR_{\varepsilon}^{0}(X)$ is the set of vertices of $X$, and
$S_{n,0}^{\varepsilon}(X)$ is generated by constant skeletal simplices. By the
base case in the proof of Theorem \ref{30}, $K_{n,0}^{\varepsilon}(X)=0$ for
$n>0$, and in degree $0$ both sides are the free abelian group generated by
the points of $X$. Hence
\[
\kappa_{\ast,n}^{0}:H_{n}^{\Delta}(VR_{\varepsilon}^{0}(X))\longrightarrow
K_{n,0}^{\varepsilon}(X)
\]
is an isomorphism for all $n$.

Now assume that $\kappa_{\ast,n}^{k-1}$ is an isomorphism in every degree. For
a fixed $n$, we compare the long exact sequence of the pair $(VR_{\varepsilon
}^{k}(X),VR_{\varepsilon}^{k-1}(X))$ with the long exact sequence associated
to
\[
0\longrightarrow S_{n,k-1}^{\varepsilon}(X)\longrightarrow S_{n,k}%
^{\varepsilon}(X)\longrightarrow S_{n}^{\varepsilon}(X^{k},X^{k-1}%
)\longrightarrow0.
\]
This gives a commutative diagram
\[
\begin{adjustbox}{max width=\textwidth}
\begin{tikzcd}[
column sep=0.50em,
row sep=3.5em
]
\cdots
\arrow[r]
&
H^\Delta_n\!\left(VR^{k-1}_\varepsilon(X)\right)
\arrow[r]
\arrow[d, "\kappa^{k-1}_{\ast,n}"']
&
H^\Delta_n\!\left(VR^k_\varepsilon(X)\right)
\arrow[r]
\arrow[d, "\kappa^k_{\ast,n}"']
&
H^\Delta_n\!\left(VR^k_\varepsilon(X), VR^{k-1}_\varepsilon(X)\right)
\arrow[r]
\arrow[d, "\kappa^{\mathrm{rel}}_{\ast,n}"']
&
H^\Delta_{n-1}\!\left(VR^{k-1}_\varepsilon(X)\right)
\arrow[r]
\arrow[d, "\kappa^{k-1}_{\ast,n}"']
&
\cdots
\\
\cdots
\arrow[r]
&
K^\varepsilon_{n,k-1}(X)
\arrow[r]
&
K^\varepsilon_{n,k}(X)
\arrow[r]
&
K^\varepsilon_n(X^k,X^{k-1})
\arrow[r]
&
K^\varepsilon_{n-1,k-1}(X)
\arrow[r]
&
\cdots
\end{tikzcd}
\end{adjustbox}
\]
We now identify the relative vertical map. On the Vietoris--Rips side,
\[
H_{n}^{\Delta}(VR_{\varepsilon}^{k}(X),VR_{\varepsilon}^{k-1}(X))=0
\]
unless $n=k$. When $n=k$, this relative group is freely generated by the
oriented $k$-simplices $[x_{0},\ldots,x_{k}]$ of $VR_{\varepsilon}(X)$.

On the skeletal side, by Theorem \ref{30}, $K_{n}^{\varepsilon}(X^{k}%
,X^{k-1})=0$ whenever $n>k$. Also, if $n<k$, then $S_{n}^{\varepsilon}%
(X^{k},X^{k-1})=0$, because an $n$-skeletal simplex has at most $n+1$ image
points and therefore cannot have image cardinality $k+1$. Hence $K_{n}%
^{\varepsilon}(X^{k},X^{k-1})=0$ for $n<k$ as well.

Thus $K_{n}^{\varepsilon}(X^{k},X^{k-1})$ can be nonzero only when $n=k$. When
$n=k$, the quotient group
\[
S_{k}^{\varepsilon}(X^{k},X^{k-1})=S_{k,k}^{\varepsilon}(X)/S_{k,k-1}%
^{\varepsilon}(X)
\]
is generated by the relative classes of $(k,\varepsilon)$-skeletal simplices
$\sigma:V^{k}\rightarrow X$ satisfying $|\sigma(V^{k})|=k+1$. Such a simplex
is precisely an ordering of a $(k+1)$-point subset $\{x_{0},\ldots
,x_{k}\}\subset X$ of diameter less than $\varepsilon$; that is, an oriented
$k$-simplex of $VR_{\varepsilon}(X)$. The boundary maps agree, and different
orderings differ by the usual sign. Therefore
\[
\kappa_{\ast,k}^{rel}:H_{k}^{\Delta}(VR_{\varepsilon}^{k}(X),VR_{\varepsilon
}^{k-1}(X))\longrightarrow K_{k}^{\varepsilon}(X^{k},X^{k-1})
\]
is an isomorphism for every $n$.

By the induction hypothesis, the left and right vertical maps in the long
exact sequence diagram are isomorphisms. Since the relative vertical map is
also an isomorphism, the Five Lemma (\cite{H}) implies that $\kappa_{\ast
,n}^{k}: H_{n}^{\Delta}(VR_{\varepsilon}^{k}(X))\longrightarrow K_{n,k}%
^{\varepsilon}(X)$ is an isomorphism for every $n$. This completes the induction.

Now fix $n$. Since simplicial homology in degree $n$ only depends on chains in
degrees $n-1$, $n$, and $n+1$, we have $H_{n}^{\Delta}(VR_{\varepsilon
}(X))\cong H_{n}^{\Delta}(VR_{\varepsilon}^{n+1}(X)).$ Similarly,
$K_{n}^{\varepsilon}(X)\cong K_{n,n+1}^{\varepsilon}(X)$, because
$S_{r,n+1}^{\varepsilon}(X)=S_{r}^{\varepsilon}(X)$ for every $r\leq n+1$.
Taking $k=n+1$ in the filtration-level isomorphism proved above gives
$H_{n}^{\Delta}(VR_{\varepsilon}^{n+1}(X))\cong K_{n,n+1}^{\varepsilon}(X).$
Combining these identifications yields $H_{n}^{\Delta}(VR_{\varepsilon
}(X))\cong K_{n}^{\varepsilon}(X)$.

Finally, the commutativity of Diagram (\ref{dr}) may be seen as follows: Then
the inclusion $VR_{\delta}(X)\hookrightarrow VR_{\varepsilon}(X)$ induces
$f_{\varepsilon\delta}:H_{n}^{\Delta}(VR_{\delta}(X))\rightarrow H_{n}%
^{\Delta}(VR_{\varepsilon}(X))$. On the skeletal side, the inclusion
$S_{n}^{\delta}(X)\hookrightarrow S_{n}^{\varepsilon}(X)$ induces
$j_{\varepsilon\delta}:K_{n}^{\delta}(X)\rightarrow K_{n}^{\varepsilon}(X)$.
Since $\kappa_{n}$ sends each ordered Vietoris--Rips simplex to the skeletal
simplex with the same ordered vertices, applying $\kappa_{n}$ and then
changing scale results in the same chain as first changing scale and then
applying $\kappa_{n}$.
\end{proof}

\begin{remark}
\label{VRR}Goldfarb's proof of this isomorphism for finite metric spaces uses
a simplicial-set comparison (\cite{G}). Our proof of Theorem \ref{VR} gives a
direct chain-level comparison by filtering the skeletal complex according to
the cardinality of the image of a simplex. This shows explicitly that the
filtration layer supported on at most $m+1$ vertices contributes no homology
above degree $m$. In particular, every class in $K_{n}^{\varepsilon}(X)$ can
be represented by a cycle whose $n$-simplices have $n+1$ distinct vertices,
although repeated-vertex simplices may still occur in bounding chains.
\end{remark}

\begin{remark}
The construction of the groups $K_{n}^{\varepsilon}(X^{n},X^{n-1}) $ is
somewhat similar to the skeletal-pair construction used for simplicial
complexes; compare the long exact sequence for the pairs $(X^{n},X^{n-1})$ in
\cite{H}. It is therefore natural to ask whether $K_{n}^{\varepsilon}(X) \cong
K_{n}^{\varepsilon}(X^{n},X^{n-1}). $ If $K_{n,n-1}^{\varepsilon}(X)=0, $ then
the long exact sequence gives an injection $j_{\ast}:K_{n}^{\varepsilon
}(X)\hookrightarrow K_{n}^{\varepsilon}(X^{n},X^{n-1}).$ The following example
shows that this injection need not be surjective.
\end{remark}

\begin{example}
Let $X=\{v_{1},v_{2},v_{3},v_{4}\}$ be the four vertices of a square, with
$d(v_{i},v_{i+1})=1,\;\;d(v_{1},v_{3})=d(v_{2},v_{4})=\sqrt{2}.$ Set
$\varepsilon=\sqrt{2}$. Thus adjacent pairs are allowable, while diagonal
pairs are not. Every subset of three distinct vertices contains a diagonal
pair. Hence every allowable $\sqrt{2}$-simplex has at most two distinct image
points, and $K_{1,1}^{\sqrt{2}}(X)=K_{1}^{\sqrt{2}}(X)=\mathbb{Z}.$ Since
$K_{1,0}^{\sqrt{2}}(X)=0\;\text{and}\;K_{0}^{\sqrt{2}}(X^{1},X^{0})=0,$the
long exact sequence of the pair $(X^{1},X^{0})$ gives
\[
0\longrightarrow K_{1}^{\sqrt{2}}(X)\overset{j_{\ast}}{\longrightarrow}%
K_{1}^{\sqrt{2}}(X^{1},X^{0})\overset{\partial}{\longrightarrow}K_{0,0}%
^{\sqrt{2}}(X)\overset{i_{\ast}}{\longrightarrow}K_{0,1}^{\sqrt{2}%
}(X)\longrightarrow0.\tag{\(\ast\)}
\]
For an adjacent pair $\{a,b\}$, the relative $1$-chains are generated by
$[ab]$ and $[ba]$. Modulo constant edges, the allowable $2$-simplices give the
single relation $[ab]+[ba]=0.$ Thus, each square edge contributes one copy of
$\mathbb{Z}$. Since no allowable $2$-simplex can involve two distinct square
edges, these relations do not mix the four edges. Therefore $K_{1}^{\sqrt{2}%
}(X^{1},X^{0})\cong\mathbb{Z}^{4}.$ Moreover, $K_{0,0}^{\sqrt{2}}%
(X)\cong\mathbb{Z}^{4},\;K_{0,1}^{\sqrt{2}}(X)=K_{0}^{\sqrt{2}}(X)\cong%
\mathbb{Z}.$ Consequently, $K_{1}^{\sqrt{2}}(X)$ is not isomorphic to
$K_{1}^{\sqrt{2}}(X^{1},X^{0}).$
\end{example}

\section{Uniformly Locally Contractible Spaces}

This section is devoted to the proof of Theorem \ref{HM2} and related topics.
We first recall a basic construction from topology in this metric space
setting. Suppose that $B=B(v,\varepsilon)$ is a contractible ball in a metric
space $X$. Recall that the cone $c(B)$ is defined to be $I\times
B/\symbol{126}$ with the quotient topology where $\symbol{126}$ identifies all
points $(0,x)$ to a single point. Now suppose that $c:I\times B\rightarrow B$
is a contraction to $v$, i.e. a continuous function such that $c(1,x)=x$ and
$c(0,x)=v$ for all $x\in B$. We define the \textit{singular cone map }%
$\psi_{B}:c(B)=I\times B/\symbol{126}=B\rightarrow B$ with $\psi
_{B}([t,x])=c(t,x)$ to be the induced map of $c$ on the quotient space $c(B)$.
If $f:A\rightarrow B$ is a function, we define the \textit{singular cone
}$\psi_{v,\varepsilon}(f)$ of $f$ to be $\psi_{B(v,\varepsilon)}%
(f):c(A)\rightarrow B(v,\varepsilon)$. In particular, if $\sigma:\Delta
^{n}\longrightarrow B(v,\varepsilon)$ is a continuous singular $n$-simplex,
then we define $\psi_{v,\varepsilon}(\sigma):C(\Delta^{n})\longrightarrow
B(v,\varepsilon)$ by $\psi_{v,\varepsilon}(\sigma)([t,u])=c(t,\sigma
(u)),\;[t,u]\in C(\Delta^{n})$. After identifying $C(\Delta^{n})$ with
$\Delta^{n+1}$, this is a continuous singular $(n+1)$-simplex. With the cone
vertex ordered first, the cone construction satisfies the boundary identity
$\partial\psi_{v,\varepsilon}(\sigma)=\sigma-\psi_{v,\varepsilon}%
(\partial\sigma).$ By linearity, for any singular chain $z$, we have
$\partial\psi_{v,\varepsilon}(z)=z-\psi_{v,\varepsilon}(\partial z).$ In
particular, if $z$ is a cycle, then $\partial z=0$, and hence $\partial
\psi_{v,\varepsilon}(z)=z.$ Regardless of the diameter of $\sigma$, the best
one can generally conclude is that $\psi_{v,\varepsilon}(\sigma)$ is a
$2\varepsilon$-simplex. This \textquotedblleft loss\textquotedblright\ will
have to be controlled via barycentric subdivision when showing that
$\rho_{\varepsilon}$ is surjective at small enough scales.

\begin{proof}
[Proof of Theorem \ref{HM2}]We begin by defining a \textquotedblleft filling
chain map\textquotedblright\ $F_{n}^{\varepsilon}:S_{n}^{\varepsilon
}(X)\rightarrow C_{n}^{\rho^{n+1}\varepsilon}(X)$ as follows. When there is no
confusion, we eliminate the superscript $\varepsilon$ and/or subscript $n$.
For any skeletal simplex $\sigma$, let $v_{\sigma}$ be any point such that the
image of $\sigma$ lies in the open ball $B(v_{\sigma},\varepsilon)$. We make
no effort concerning continuity of the map $\sigma\mapsto v_{\sigma}$ (which
is generally impossible anyway). We can get away with this, roughly speaking,
because with skeletal homology, \textquotedblleft continuity at relatively
small scales is not important\textquotedblright. Let $F_{0}(\sigma)=\sigma$,
and iteratively define $F_{i}(\sigma)=\psi_{v_{\sigma},\rho^{i}\varepsilon
}(F_{i-1}(\partial\sigma))$. By construction, the image of $F_{i}(\sigma)$
lies in $B(v_{\sigma},\rho^{i}\varepsilon)$, and therefore
$\operatorname{diam}\bigl(\operatorname{im}F_{i}(\sigma)\bigr)<\rho
^{i+1}\varepsilon$. Since $F_{i-1}$ is a chain map, we have $\partial
F_{i-1}(\partial\sigma)=F_{i-2}(\partial^{2}\sigma)=0.$ Thus $F_{i-1}%
(\partial\sigma)$ is a cycle. Therefore, using the cone-boundary identity,
$\partial\psi_{v_{\sigma},\rho^{i}\varepsilon}(z)=z-\psi_{v_{\sigma},\rho
^{i}\varepsilon}(\partial z)$, with $z=F_{i-1}(\partial\sigma)$, we obtain
\begin{equation}
\partial F_{i}(\sigma)=\partial\psi_{v_{\sigma},\rho^{i}\varepsilon}%
(F_{i-1}(\partial\sigma))=F_{i-1}(\partial\sigma)\label{three}%
\end{equation}
and therefore the linear extension of $F_{i}$ is a chain map, which we will
usually denote simply by $F$.

Since the simplices in $C_{S,q}^{\rho^{n+1}\varepsilon}(X)$ are continuous,
the usual barycentric subdivision operator and chain homotopy restrict to maps
$S:C_{S,q}^{\rho^{n+1}\varepsilon}(X) \longrightarrow C_{S,q}^{\rho
^{n+1}\varepsilon}(X) $ and $T:C_{S,q}^{\rho^{n+1}\varepsilon}(X)
\longrightarrow C_{S,q+1}^{\rho^{n+1}\varepsilon}(X), $ satisfying $\partial
T+T\partial=\mathbb{I}-S. $ For $m\geq0$, let $D_{m}=\sum_{0\leq i<m}TS^{i}. $
For each singular $q$-simplex $\sigma\in C_{S,q}^{\rho^{n+1}\varepsilon}(X)$,
let $m(\sigma)$ be the least nonnegative integer such that $S^{m(\sigma
)}\sigma\in C_{S,q}^{\varepsilon}(X), $ and define $D(\sigma)=D_{m(\sigma
)}(\sigma), $ extending linearly. Thus $D:C_{S,q}^{\rho^{n+1}\varepsilon}(X)
\longrightarrow C_{S,q+1}^{\rho^{n+1}\varepsilon}(X). $ Define $R=\mathbb{I}%
-\partial D-D\partial. $ As in the proof of Proposition~2.21 of \cite{H}, $R$
is a chain map whose image lies in $C_{S,n}^{\varepsilon}(X)$. Hence we may
regard $R:C_{S,n}^{\rho^{n+1}\varepsilon}(X) \longrightarrow C_{S,n}%
^{\varepsilon}(X). $%

\[
\begin{tikzcd}[column sep=2.0em, row sep=2.0em]
S_n^\varepsilon(X)
\arrow[r, "F"]
&
C_{S,n}^{\rho^{n+1}\varepsilon}(X)
\arrow[r, "r_\#"]
\arrow[d, "R"']
&
S_n^{\rho^{n+1}\varepsilon}(X)
\\
&
C_{S,n}^{\varepsilon}(X)
\arrow[r, "r_\#"]
&
S_n^\varepsilon(X)
\arrow[u, hook]
\end{tikzcd}
\]
By the preceding remark, $r_{\#}(F(\sigma))=\sigma$, meaning that
$r_{\#}\circ F:S_{n}^{\varepsilon}(X)\rightarrow S_{n}^{\varepsilon}(X)$ is
the identity map. By definition, the image of the induced map $r_{\ast}$ of
the bottom $r_{\#}$ is $H_{n}^{\varepsilon}(X)$. Therefore, we have the
following diagram of induced maps:
\[
\begin{tikzcd}[column sep=2.0em, row sep=2.0em]
K_n^\varepsilon(X)
\arrow[r, "F_\ast"]
&
H_{S,n}^{\rho^{n+1}\varepsilon}(X)
\arrow[r, "r_\ast"]
\arrow[d, "R_\ast"']
&
K_n^\varepsilon(X)
\\
&
H_{S,n}^{\varepsilon}(X)
\arrow[r, "r_\ast"]
&
H_n^\varepsilon(X)
\arrow[u, hook]
\end{tikzcd}
\]

Define a chain homotopy $B=r_{\#}\circ D\circ F:S_{n}^{\varepsilon
}(X)\rightarrow S_{n}^{\varepsilon}(X)$. Then
\[
\partial B+B\partial=\partial\left(  r_{\#}\circ D\circ F\right)  +\left(
r_{\#}\circ D\circ F\right)  \partial
\]%
\[
=r_{\#}\circ\left(  \partial D+D\partial\right)  \circ F=r_{\#}\circ\left(
\mathbb{I}-R\right)  \circ F
\]%
\[
=r_{\#}\circ F-r_{\#}\circ R \circ F
\]
This shows that the induced maps $r_{\ast}\circ F_{\ast}$ and $r_{\ast}\circ
R_{\ast}\circ F_{\ast}$ are the same, and we conclude that $H_{n}%
^{\varepsilon}(X)=K_{n}^{\varepsilon}(X)$.

Now suppose that $[c]\in\ker\rho_{\varepsilon}$, i.e. we may choose $c$ to be
a continuous cycle such that $r_{\#}(c)=\partial d$ for some skeletal
$\varepsilon\,$-chain $d$. Here $c\in C_{S,n}^{\varepsilon}(X)$ and $d\in
S_{n+1}^{\varepsilon}(X).$ For the injectivity argument, the filling
construction used above must be carried out through dimension $n+1$. Thus we
have homomorphisms $F_{q}:S_{q}^{\varepsilon}(X)\longrightarrow C_{q}%
(X),\;0\leq q\leq n+1,$ such that $\partial F_{q}=F_{q-1}\partial$ for $1\leq
q\leq n+1$. Carrying out the preceding inductive filling construction
simultaneously for small continuous simplices, we may choose homomorphisms
$P_{q}:C_{S,q}^{\varepsilon}(X)\longrightarrow C_{q+1}(X),\;0\leq q\leq n,$
with $P_{-1}=0$, such that $\partial P_{q}+P_{q-1}\partial=\operatorname{id}%
-F_{q}r_{\#}.$ Indeed, inductively, $P_{q}(\sigma)$ is obtained by filling the
cycle $\sigma-F_{q}(r_{\#}\sigma)-P_{q-1}(\partial\sigma)$ inside the same
sufficiently small contractible ball.

Since $c$ is a cycle, the above equation gives $\partial P_{n}(c)=c-F_{n}%
(r_{\#}(c)).$ On the other hand, since $r_{\#}(c)=\partial d$, with $F$ we
have $\partial F_{n+1}(d)=F_{n}(\partial d)=F_{n}(r_{\#}(c))$. Adding these
two, we obtain $\partial\bigl(P_{n}(c)+F_{n+1}(d)\bigr)=c-F_{n}(r_{\#}%
(c))+F_{n}(r_{\#}(c))=c.$ Thus $c$ is a continuous singular boundary.
Therefore $[c]=0\in H_{n}(X)$, completing the proof.
\end{proof}
\begin{remark}
Although $F_n(\sigma)$ is written as a sum of singular simplices, the
simplices arising from the recursive cone construction agree on their
common faces and hence, by the gluing lemma, determine a single
continuous simplex
\(\widehat F_n(\sigma):\Delta^n\to X.
\) Equivalently, $F_n(\sigma)$ is obtained by subdividing this glued
simplex. Since the gluing preserves the original vertices,
\(\widehat F_n(\sigma)|_{V_n}=\sigma,
\) and therefore
\(r_\#\bigl(\widehat F_n(\sigma)\bigr)=\sigma.
\) We henceforth use $F$ to denote this glued representative.
\end{remark}
\begin{example}
Any Euclidean $2$-cone with cone angle less than $\pi/3$ is not uniformly
locally contractible, but has finite contractibility ratio. In \cite{SH2} we
will show that Theorem \ref{HM2} applies to many spaces that have a local cone
structure even though they may not be uniformly locally contractible.
\end{example}

As far as we know, Theorem \ref{HM2} implies qualitatively, and sometimes
quantitatively, all of the main Hausmann-type theorems, considered as
statements about homology. These include the Adamaszek-Adams-Frick Theorem 4.2
of \cite{AA2}, about Riemannian manifolds with uniformly bounded curvature and
a uniform positive lower bound on convexity radius. We note that this theorem
is misstated in the Introduction as Theorem Main (4), which omits the uniform
bound on convexity radius. See condition (i) on p. 605, which is assumed for
Theorem 4.2. Without this assumption, Example \ref{rev} is a counterexample.
There is also Theorem 4.6 in \cite{AM}, which applies to subsets of Euclidean
space of positive reach. But as observed by Federer (Remark 4.15, \cite{Fed}),
subsets of Euclidean space of positive reach are uniformly locally
contractible. In a very recent preprint, Hausmann's theorem was extended to
spaces of curvature bounded above by $\kappa$ (CBA($\kappa$) spaces) by
Oudot-Waas (\cite{O2}). Since compact CBA spaces are uniformly locally
contractible, Theorem \ref{HM} also implies this result. Moreover, in a
CBA($\kappa$) space with an upper curvature bound, it is not hard to argue
inductively that skeletal simplices in a neighborhood in which Alexandrov's
comparisons are valid have continuous extensions of the same diameter. This
means that our short proof of Theorem \ref{HM2} can be modified to show that
$\rho_{\varepsilon}:H_{n}(X)\rightarrow K_{n}^{\varepsilon}(X)$ is true when
$\varepsilon>0$ is less than any uniform \textquotedblleft comparison radius"
of a CBA($\kappa$) space. We leave the details to the reader. 

Wheter the stronger conclusion of Hausmann's Theorem about the existence of a
homotopy equivalence between $X$ and $VR_{\varepsilon}(X)$ at small enough
scales, under the assumptions of Theorem \ref{HM2}, is still open. The fact
that the scale in Theorem \ref{HM2} depends on $n$ leads us to suspect that
the answer may be negative.

\section{Stability}

\label{stabsec}In this last section we prove Theorem \ref{induced}. The proof
fittingly uses a discrete variant of the \textquotedblleft prism
construction\textquotedblright, which is used to show the classical theorem
that homotopic maps induce the same homomorphism on homology (Theorem 2.10 in
\cite{H}). In the classical setting one must subdivide the prism $\Delta
^{n}\times I$ into $(n+1)$-simplices and use a homotopy to map them
continuously into the space. In our setting the \textquotedblleft
prism\textquotedblright\ over a skeletal simplex is simply the finite set
$V^{n}\times\{0,1\}$, and the role of the homotopy is played by the uniform
closeness of $f$ and $g$.

\begin{proof}
[Proof of Theorem \ref{induced}]Since $\omega$ is a common modulus of
continuity for $f$ and $g$, every skeletal $n$-simplex $\sigma:V^{n}%
\rightarrow X$ whose image has diameter less than $\omega(\varepsilon)$ is
carried by both $f$ and $g$ to skeletal $n$-simplices whose images have
diameter less than $\varepsilon$. Thus
\[
f_{\#},g_{\#}:S_{n}^{\omega(\varepsilon)}(X)\longrightarrow S_{n}%
^{\varepsilon}(Y)\subset S_{n}^{\varepsilon+\delta}(Y)
\]
are well-defined chain maps.

For a skeletal simplex $\sigma\in S_{n}^{\omega(\varepsilon)}(X)$ write
$f_{i}=f(\sigma(v_{i}))$ and $g_{i}=g(\sigma(v_{i}))$. As in the definition of
the VR complex, we denote by $[y_{0},\ldots,y_{m}]$ the skeletal $m$-simplex
$\tau$ with $\tau(v_{k})=y_{k}$; repetitions among the $y_{k}$ are allowed.
For $r=0,\ldots,n$ define
\[
p_{r}(\sigma)=[f_{0},\ldots,f_{r},g_{r},\ldots,g_{n}]\text{,}%
\]
which is a function from $V^{n+1}$ to $Y$, and let
\[
P_{n}(\sigma)=\sum_{r=0}^{n}(-1)^{r}p_{r}(\sigma)\text{,}%
\]
extended linearly to $S_{n}^{\omega(\varepsilon)}(X)$, with $P_{-1}=0$.

We first check that each $p_{r}(\sigma)$ is an $(n+1,\varepsilon+\delta
)$-simplex. Two points on the $f$-side satisfy $d(f_{i},f_{j})<\varepsilon$,
since $f\circ\sigma$ is an $\varepsilon$-simplex, and likewise for two points
on the $g$-side. For a mixed pair $f_{i},g_{j}$ we have
\[
d(f_{i},g_{j})\leq d(f_{i},f_{j})+d(f_{j},g_{j})<\varepsilon+\delta
\]
since $\left\vert f-g\right\vert <\delta$. Since $V^{n+1}$ is finite, the
image of $p_{r}(\sigma)$ has diameter less than $\varepsilon+\delta$. Hence
$P_{n}:S_{n}^{\omega(\varepsilon)}(X)\rightarrow S_{n+1}^{\varepsilon+\delta
}(Y)$.

We now prove the prism identity
\begin{equation}
\partial P_{n}+P_{n-1}\partial=g_{\#}-f_{\#}\text{.}\label{prism}%
\end{equation}
By linearity it suffices to check (\ref{prism}) on a skeletal simplex
$\sigma$. The computation is exactly the one in the proof of Theorem 2.10 in
\cite{H}, but since it is purely combinatorial we include it for completeness.
We have
\[
\partial P_{n}(\sigma)=\sum_{i\leq r}(-1)^{r}(-1)^{i}[f_{0},\ldots
,\widehat{f_{i}},\ldots,f_{r},g_{r},\ldots,g_{n}]\]
\[+\sum_{i\geq r}%
(-1)^{r}(-1)^{i+1}[f_{0},\ldots,f_{r},g_{r},\ldots,\widehat{g_{i}}%
,\ldots,g_{n}]\text{.}%
\]
Consider first the terms with $i=r$. The term $[f_{0},\ldots,f_{r-1}%
,g_{r},\ldots,g_{n}]$ appears in the first sum (with $i=r$) with sign $+1$ and
in the second sum (with $i=r-1$) with sign $-1$, so these terms cancel for
$1\leq r\leq n$. The only remaining terms with $i=r$ are the term $[g_{0}%
,\ldots,g_{n}]=g_{\#}(\sigma)$ from the first sum with $i=r=0$ and the term
$-[f_{0},\ldots,f_{n}]=-f_{\#}(\sigma)$ from the second sum with $i=r=n$.

On the other hand, if $\sigma^{i}$ denotes the $i^{th}$ face of $\sigma$, then
\[
P_{n-1}(\partial\sigma)=\sum_{i}(-1)^{i}P_{n-1}(\sigma^{i})=\sum
_{i<r}(-1)^{i}(-1)^{r-1}[f_{0},\ldots,\widehat{f_{i}},\ldots,f_{r}%
,g_{r},\ldots,g_{n}]\]
\[+\sum_{i>r}(-1)^{i}(-1)^{r}[f_{0},\ldots,f_{r}%
,g_{r},\ldots,\widehat{g_{i}},\ldots,g_{n}]\text{,}%
\]
where the shift in the sign of the first sum arises because $v_{r}$ is the
$(r-1)^{st}$ vertex of the face $\sigma^{i}$ when $i<r$. Comparing with the
terms of $\partial P_{n}(\sigma)$ with $i\neq r$ shows that those terms sum to
precisely $-P_{n-1}(\partial\sigma)$. Altogether,
\[
\partial P_{n}(\sigma)=g_{\#}(\sigma)-f_{\#}(\sigma)-P_{n-1}(\partial
\sigma)\text{,}%
\]
which is (\ref{prism}).

Now let $[c]_{\omega(\varepsilon)}\in K_{n}^{\omega(\varepsilon)}(X)$, where
$c$ is an $\omega(\varepsilon)$-cycle. Since $\partial c=0$, (\ref{prism})
gives
\[
g_{\#}(c)-f_{\#}(c)=\partial P_{n}(c)+P_{n-1}(\partial c)=\partial P_{n}(c)
\]
in $S_{n}^{\varepsilon+\delta}(Y)$. That is, $f_{\#}(c)$ and $g_{\#}(c)$
differ by the boundary of an $(\varepsilon+\delta)$-chain, and therefore
\[
f_{\ast}([c]_{\omega(\varepsilon)})=[f_{\#}(c)]_{\varepsilon+\delta}%
=[g_{\#}(c)]_{\varepsilon+\delta}=g_{\ast}([c]_{\omega(\varepsilon)})\text{.}%
\]
Since $[c]_{\omega(\varepsilon)}$ was arbitrary, $f_{\ast}=g_{\ast}%
:K_{n}^{\omega(\varepsilon)}(X)\rightarrow K_{n}^{\varepsilon+\delta}(Y)$.
\end{proof}

\begin{acknowledgement}
The second author gratefully acknowledges the support of the American
Institute of Mathematics for his participation in the AIM program
\textquotedblleft Discrete and Combinatorial Homotopy
Theory,\textquotedblright held March 13--17, 2023. This program fostered new
research connections and motivated his study of homology at scale. We thank
Ernesto Ugona several valuable discussions, as well as Antonio Rieser and
Nikola Mili\'{c}evi\'{c} for helpful comments.
\end{acknowledgement}

\end{document}